\documentclass{siamart1116}

\usepackage{mathtools}
\usepackage{graphicx}
\usepackage{amsfonts}
\usepackage{amsmath}
\usepackage{tikz}
\usepackage{pgf}
\usetikzlibrary{decorations.pathreplacing,matrix}
\usetikzlibrary{arrows.meta,bending,matrix,positioning}

\usepackage{amsmath}
\usepackage{amssymb}
\usepackage{epstopdf}
\usepackage{geometry} 
\usepackage{float}
\usepackage{algorithm}
\usepackage[noend]{algpseudocode}
\usepackage{color}
\usepackage{todonotes}
\usepackage{array}
\usepackage{listings}
\usepackage{hyperref}
\usepackage{bm}

\ifpdf
  \DeclareGraphicsExtensions{.eps,.pdf,.png,.jpg}
\else
  \DeclareGraphicsExtensions{.eps}
\fi

\newtheorem{remark}[theorem]{Remark}

\newcommand{\TheTitle}{Numerical analysis of the maximal attainable accuracy in communication hiding pipelined Conjugate Gradient methods} 
\newcommand{\TheAuthors}{Siegfried Cools}

\headers{Accuracy analysis for pipelined Conjugate Gradients}{S.~Cools}

\title{{\TheTitle}\thanks{Submitted to the editors on \today.
\funding{This work was funded by Research Foundation Flanders (FWO).}}}

\author{
  Siegfried Cools\thanks{Applied Mathematics Group, Department of Mathematics and Computer Science, University of Antwerp, Building G, Middelheimlaan 1, 2020 Antwerp, Belgium.} 
}

\usepackage{amsopn}

\ifpdf
\hypersetup{
  pdftitle={\TheTitle},
  pdfauthor={\TheAuthors}
}
\fi

\begin{document}

\maketitle

\begin{keywords}
  Pipelined Krylov subspace methods, Parallel performance, Exascale computations, Global communication, Latency hiding, Conjugate Gradients, Numerical stability, Inexact computations, Maximal attainable accuracy.
\end{keywords}

\vspace{-0.3cm}

\begin{AMS}
  65F10, 65N12, 65G50, 65Y05, 65N22.
\end{AMS}

\begin{abstract}
	Krylov subspace methods are widely known as efficient algebraic methods for solving large scale linear systems. However, on massively parallel hardware the performance of these methods is typically limited by communication latency rather than floating point performance. With HPC hardware advancing towards the exascale regime the gap between computation (i.e. flops) and communication (i.e. internode communication, as well as data movement within the memory hierarchy) keeps steadily increasing, imposing the need for scalable alternatives to traditional Krylov subspace methods. One such approach are the so-called pipelined Krylov subspace methods, which reduce the number of global synchronization points and overlap global communication latency with local arithmetic operations, thus `hiding' the global reduction phases behind useful computations. To obtain this overlap the traditional Krylov subspace algorithm is reformulated by introducing a number of auxiliary vector quantities, which are computed using additional recurrence relations. Although pipelined Krylov subspace methods are equivalent to traditional Krylov subspace methods in exact arithmetic, local rounding errors induced by the multi-term recurrence relations in finite precision may in practice affect convergence significantly. This numerical stability study aims to characterize the effect of local rounding errors on attainable accuracy in various pipelined versions of the popular Conjugate Gradient method. Expressions for the gaps between the true and recursively computed variables that are used to update the search directions in the different CG variants are derived. Furthermore, it is shown how these results can be used to analyze and correct the effect of local rounding error propagation on the maximal attainable accuracy of pipelined CG methods. The analysis in this work is supplemented by various numerical experiments that demonstrate the numerical behavior of the pipelined CG methods.
\end{abstract}

\section{Introduction} \label{sec:introduction}

Krylov subspace methods \cite{greenbaum1997iterative,liesen2012krylov,meurant1999computer,saad2003iterative,van2003iterative} have been used for decades as efficient iterative solution methods for linear systems. This paper considers the problem of solving algebraic linear systems of the form $Ax= b$, where $A$ is a real or complex non-singular square $n \times n$ matrix and the right-hand side vector $b$ correspondingly has length $n$. Given an initial guess $x_0$ for the solution $x$ and an initial residual $r_0 = b-Ax_0$, Krylov subspace methods construct a series of approximate solutions $x_i$ that lie in the $i$-th Krylov subspace
\[
	x_i \in x_0 + \mathcal{K}_i(A,r_0) = \text{span}\{r_0, Ar_0, \ldots , A^{i-1} r_0\}, \qquad i = 1,2,\ldots,
\]
using some orthogonality constraint that differentiates the various Krylov subspace methods. For problems with symmetric (or Hermitian) positive definite (SPD) matrices $A$ -- which are the main focus of this work -- one of the most basic yet widely used Krylov subspace methods is the (preconditioned) Conjugate Gradient (CG) method, that dates back to the original 1952 paper \cite{hestenes1952methods}, in which the orthogonality constraint boils down to minimizing the $A$-norm of the error over the Krylov subspace, i.e.,
\[
  \|x-x_i\|_A = \min_{y \in x_0 + \mathcal{K}_i(A,r_0)} \|x-y\|_A,
\] 
where the energy norm is defined as $\|x\|_A = (Ax,x)^{1/2} = (x,Ax)^{1/2}$ with the SPD matrix $A$.

A clear trend in current (petascale) and future (exascale) high performance computing hardware consists of up-scaling the number of parallel compute nodes as well as the number of processors per node \cite{dongarra2011international}. Compute nodes are generally interconnected using fast InfiniBand connections. However, the increasing gap between computational performance and memory/interconnect latency implies that on HPC hardware data movement is much more expensive than floating point operations (flops), both with respect to execution time as well as energy consumption \cite{dongarra2011international,fuller2011future}. As such, reducing time spent in moving data and/or waiting for data will be essential for exascale applications. 

Although Krylov subspace algorithms are traditionally highly efficient with respect to the number of flops that are required to obtain a sufficiently precise approximate solution $x_i$, they are not optimized towards minimizing communication. The HPCG benchmark \cite{dongarra2013toward,dongarra2015hpcg}, for example, shows that Krylov subspace methods are able to attain only a small fraction of the machine peak performance on large scale hardware. Since the matrix $A$ is often sparse and generally requires only limited communication between neighboring processors, the primary bottleneck for parallel execution is typically not the (sparse) matrix-vector (\textsc{spmv}) product. Instead, parallel efficiency stalls due to communication latency caused by global synchronization phases when computing the dot-products required for the orthogonalization procedures in the Krylov subspace algorithm.

A variety of interesting research branches on reducing or eliminating the synchronization bottleneck in Krylov subspace methods has emerged over the last decades. Based on the earliest ideas of communication reduction in Krylov subspace methods \cite{strakovs1987effectivity,chronopoulos1989s,d1992reducing,demmel1993parallel,erhel1995parallel,de1995reducing,chronopoulos1996parallel}, a number of methods that aim to eliminate global synchronization points has recently been (re)introduced. These include hierarchical Krylov subspace methods \cite{mcinnes2014hierarchical}, enlarged Krylov subspace methods \cite{grigori2016enlarged}, an iteration fusing Conjugate Gradient method \cite{zhuang2017iteration}, $s$-step Krylov subspace methods (also called ``communication avoiding'' Krylov subspace methods) \cite{chronopoulos2010block,carson2013avoiding,carson2014residual,imberti2017varying}, and pipelined Krylov subspace methods (also referred to as ``communication hiding'' Krylov subspace methods) \cite{ghysels2013hiding,ghysels2014hiding,sanan2016pipelined,eller2016scalable,yamazaki2017improving}. 
 The latter aim not only to reduce the number of global synchronization points in the algorithm, but also overlap global communication with useful (mostly local) computations such as \textsc{spmv}s ($y \leftarrow Ax$) and \textsc{axpy}s (
$y \leftarrow \alpha x + y$). In this way idle core time is reduced by simultaneously performing the synchronization phase and the independent compute-bound calculations. When the time required by one global communication phase approximately equals the computation time for the \textsc{spmv}, the pipelined CG method from \cite{ghysels2014hiding}, 
denoted as p-CG in this work, allows for a good overlap and thus 
features significantly reduced
total time to solution compared to the classic CG method. In heavily communication-bound scenarios where a global reduction takes significantly longer than one \textsc{spmv}, deeper pipelines allow to overlap the global reduction phase with the computational work of multiple \textsc{spmv}s. The concept of deep pipelines was introduced by Ghysels et al.~\cite{ghysels2013hiding} for the Generalized Minimal Residual (GMRES) method, and was recently extended to the 
CG method \cite{cornelis2017communication}. 

The advantages of using pipelined (and other) communication reducing Krylov subspace methods from a performance point of view have been illustrated in many of the aforementioned works. However, reorganizing the traditional Krylov subspace algorithm into a communication reducing variant typically introduces unwanted issues with the numerical stability of the algorithm. In exact arithmetic pipelined Krylov subspace methods produce a series of iterates identical to the classic CG method. However, in finite precision arithmetic their behavior can differ significantly as local rounding errors may induce a decrease in attainable accuracy and a delay of convergence. The impact of finite precision round-off errors on numerical stability of classic CG has been extensively studied in a number of papers among which \cite{greenbaum1989behavior,greenbaum1992predicting,greenbaum1997estimating,gutknecht2000accuracy,strakovs2002error,strakovs2005error,meurant2006lanczos,gergelits2014composite}. 

In communication reducing CG variants the effects of local rounding errors are significantly amplified. We refer to our previous work \cite{cools2018analyzing} and the paper \cite{carson2016numerical} by Carson et al.~for a (historic) overview of the numerical stability analysis of pipelined Conjugate Gradient methods. In the present study we focus on analyzing the recently introduced deep pipelined Conjugate Gradient method \cite{cornelis2017communication}, denoted by p($\ell$)-CG.\footnote{Note for a proper understanding that the p($\ell$)-CG method in \cite{cornelis2017communication} was not derived from the p-CG method \cite{ghysels2014hiding}, but is rather based on similar principles as the p($\ell$)-GMRES method \cite{ghysels2013hiding}. Although both p-CG and p($\ell$)-CG are pipelined variants of the CG algorithm suffering from local rounding error propagation, the exact underlying mechanism by which error propagation occurs differs between these methods, as analyzed in Sections \ref{sec:pipecg}-\ref{sec:pipelcg}  of this work.} In \cite{cornelis2017communication} it was observed that the numerical accuracy attainable by the p($\ell$)-CG method may be reduced drastically for larger pipeline lengths $l$. Similar observations have been made for other classes of communication reducing methods, see e.g.~the studies \cite{chronopoulos1989s,carson2014residual} for the influence of the $s$-step parameter on the numerical stability of communication avoiding Krylov subspace methods. 
The performance of pipelined Krylov subspace methods with respect to HPC system noise has recently been analyzed in a stochastic framework, see \cite{morgan2016stochastic}.

The sensitivity of the p($\ell$)-CG Krylov subspace method to local rounding errors may induce an obstacle for the practical useability and reliability of this pipelined method. It is therefore vital to thoroughly understand and -- if possible -- counteract the negative effects of local rounding error propagation on the numerical stability of the method. 
The aim of this study is two-fold:
\vspace{0.1cm}
\begin{enumerate}
	\item to establish a theoretical framework for the behavior of local rounding errors that explains the observed loss of attainable accuracy in pipelined CG methods;
	\vspace{0.1cm}
	\item to exemplify possible -- preferably easy-to-implement and cost-efficient -- countermeasures which can be used to improve the attainable accuracy of these methods when required.
\end{enumerate}
\vspace{0.1cm}
This paper analyzes the behavior of local rounding errors that stem from the multi-term recurrence relations in the p($\ell$)-CG algorithm, and compares to 
recent related work in \cite{carson2016numerical,cools2018analyzing} on 
the length-one pipelined CG method \cite{ghysels2014hiding}. We explicitly characterize the propagation of local rounding errors throughout the algorithm and discuss the influence of the pipelined length $l$ and the choice of the Krylov basis on the maximal attainable accuracy of the method. Furthermore, based on the error analysis, a possible approach for stabilizing the p($\ell$)-CG method is suggested near the end of the manuscript. The analysis in this paper is limited to the effect of local rounding errors in the multi-term recurrence relations on attainable accuracy; a discussion of the loss of orthogonality \cite{greenbaum1992predicting,gutknecht2000accuracy} and consequential delay of convergence in pipelined CG variants is beyond the scope of this work.

The paper is structured as follows. Section \ref{sec:pipecg} briefly summarizes the existing numerical rounding error analysis for classic CG and pipelined CG (p-CG), while meanwhile introducing the notation that will be used throughout the manuscript. Section \ref{sec:pipelcg} is devoted to analyzing the behavior of local rounding errors that stem from the recurrence relations in p($\ell$)-CG. In Section \ref{sec:further} the numerical analysis is extended by establishing practical bounds for the operator that governs the error propagation. This allows to substantiate the impact of the pipeline length $l$, the iteration index $j$ and the choice of the Krylov basis on attainable accuracy. Section \ref{sec:countermeasures} presents a countermeasure to reduce the impact of local rounding errors on final attainable accuracy. This stabilization technique results directly from the error analysis; however, it comes at the cost of an increase in the algorithm's computational cost. In Section \ref{sec:numerical} we present some numerical experiments to verify and substantiate the numerical analysis in this work. Finally, the paper is concluded in Section \ref{sec:conclusions}.

\section{Analyzing the effect of local rounding errors in classic CG and p-CG} \label{sec:pipecg}

We first present an overview of the classic rounding error analysis of the standard CG method, Algorithm \ref{algo:CG}, which is treated in a broad variety of publications, see e.g.~\cite{greenbaum1992predicting,strakovs2002error,strakovs2005error,meurant2006lanczos}. In addition, we summarize the rounding error analysis of the length-one pipelined CG method (denoted as `p-CG'), Algorithm \ref{algo:PIPECG}, see Ghysels and Vanroose \cite{ghysels2014hiding}. This analysis was recently presented in \cite{cools2018analyzing,carson2016numerical}. 

\subsection{Behavior of local rounding errors in classic CG}


\begin{algorithm}[t]
{\small
  \caption{Conjugate Gradient method (CG) \hfill \textbf{Input:} $A$, $b$, $x_0$, $m$, $\tau$} \label{algo:CG}
  \begin{algorithmic}[1]
    \State $r_0 := b - Ax_0$;
		\State $p_0 := r_0$; 
    \For{$i = 0, \dots, m$}
    \State $s_i := Ap_{i}$;
    \State $\alpha_{i} := \left( r_i, r_i \right) / \left( s_i, p_i \right)$;
		\If{$\sqrt{\left( r_i, r_i \right)}/\|r_0\| < \tau$}
		\State RETURN;
		\EndIf
		\State \textbf{end if} 
    \State $x_{i+1} := x_i + \alpha_{i} p_i$;
    \State $r_{i+1} := r_i - \alpha_{i} s_i$;
    \State $\beta_{i+1} := \left( r_{i+1}, r_{i+1} \right) / \left( r_i, r_i \right)$;
    \State $p_{i+1} := r_{i+1} + \beta_{i+1} p_i$;
    \EndFor
		\State \textbf{end for} 
  \end{algorithmic}
}
\end{algorithm}

In the classic Conjugate Gradient method, Algorithm \ref{algo:CG}, the recurrence relations for the approximate solution $x_{j+1}$, the residual $r_{j+1}$, and the search direction $p_{j+1}$ are given in exact arithmetic by 
\begin{equation} \label{eq:recs_xandr_exact}
	x_{j+1} = x_j + \alpha_{j} p_j , \qquad 
	r_{j+1} = r_j - \alpha_{j} s_j , \qquad 
	p_{j+1} = r_{j+1} - \beta_{j+1} p_j ,
\end{equation}
respectively, where $s_j = A p_j$. In finite precision arithmetic the recurrence relations for these vectors do not hold exactly but are contaminated by local rounding errors in each iteration. 
We differentiate between exact variables (e.g.~the exact residual $r_j = b-Ax_j$) and the actually computed variables (e.g.~the finite precision residual $\bar{r}_j$) by using a notation with bars for the computed quantities. 

We use the classic model for floating point arithmetic with machine precision $\epsilon$. The round-off error on scalar multiplication, vector summation, \textsc{spmv} application and dot-product computation on an $n$-by-$n$ matrix $A$, length $n$ vectors $v$, $w$ and a scalar number $\alpha$ are respectively bounded by
\begin{align*}
	\| \alpha v - \text{fl}(\alpha v) \| &\leq \| \alpha v \| \, \epsilon =  |\alpha| \, \|v\| \, \epsilon, &\qquad
	\| v + w - \text{fl}(v + w) \| &\leq (\|v\| + \|w\|) \, \epsilon,\\
	\| Av - \text{fl}(Av) \| &\leq \mu\sqrt{n} \, \|A\| \, \|v\| \, \epsilon, &\qquad
	| \left( v,w \right) - \text{fl}(\,\left(v,w \right)\,) | &\leq n \, \|v\| \, \|w\| \epsilon,
\end{align*}
where $\text{fl}(\cdot)$ indicates the finite precision floating point representation, $\mu$ is the maximum number of nonzeros in any row of $A$, and the norm $\|\cdot\|$ represents the Euclidean 2-norm in this manuscript. 

Considering the recurrence relations for the approximate solution $\bar{x}_{j+1}$, the residual $\bar{r}_{j+1}$ and the search direction $\bar{p}_{j+1}$ computed by the CG algorithm in the finite precision framework yields
\begin{equation} \label{eq:recs_xandr}
	\bar{x}_{j+1} = \bar{x}_j + \bar{\alpha}_j \bar{p}_j + \xi_{j+1}^{\bar{x}} , \qquad 
	\bar{r}_{j+1} = \bar{r}_j - \bar{\alpha}_j \bar{s}_j + \xi_{j+1}^{\bar{r}} , \qquad
	\bar{p}_{j+1} = \bar{r}_{j+1} - \bar{\beta}_{j+1} \bar{p}_j + \xi_{j+1}^{\bar{p}} ,
\end{equation}
where $\xi_{j+1}^{\bar{x}}$, $\xi_{j+1}^{\bar{r}}$ and $\xi_{j+1}^{\bar{p}}$ represent local rounding errors, and where $\bar{s}_j = A \bar{p}_j$. We refer to the analysis in \cite{carson2016numerical,cools2018analyzing} for bounds on the norms of these local rounding errors.

It is well-known that local rounding errors in the classic CG algorithm are accumulated but not amplified throughout the algorithm. This is concluded directly from computing the gap on the residual $f_j = (b-A\bar{x}_j) - \bar{r}_j$. With this notation it follows from \eqref{eq:recs_xandr} that
\begin{equation} \label{eq:f_CG}
  f_{j+1} =  f_j - A\xi_{j+1}^{\bar{x}} - \xi_{j+1}^{\bar{r}} = f_0 - \sum_{k=0}^{j} ( A  \xi_{k+1}^{\bar{x}} + \xi_{k+1}^{\bar{r}} ).
\end{equation}
Hence in each iteration local rounding errors of the form $A\xi_{j+1}^{\bar{x}} + \xi_{j+1}^{\bar{r}}$ add to the gap on the residual.

By introducing the matrix notation $\mathcal{F}_{j+1} = [f_0, \ldots, f_j]$ for the residual gaps in the first $j+1$ iterations and by analogously defining $\Theta_j^{\bar{x}} = [0,-\xi_1^{\bar{x}},\ldots, -\xi_{j-1}^{\bar{x}}]$ and $\Theta_j^{\bar{r}} = [f_0,-\xi_1^{\bar{r}},\ldots, -\xi_{j-1}^{\bar{r}}]$ for the local rounding errors, expression \eqref{eq:f_CG} can be formulated as
\begin{equation} \label{eq:f_CG_mat}
  \mathcal{F}_{j+1} = (A \Theta_{j+1}^{\bar{x}} + \Theta_{j+1}^{\bar{r}}) \, U_{j+1},
\end{equation}
where $U_{j+1}$ is a $(j+1) \times (j+1)$ upper triangular matrix of ones. 
No amplification of local rounding errors occurs; indeed, local rounding errors are merely accumulated in the classic CG algorithm.

\subsection{Propagation of local rounding errors in p-CG} \label{sec:rounding_pcg}


\begin{algorithm}[t]
{\small
  \caption{Pipelined Conjugate Gradient method (p-CG) \hfill \textbf{Input:} $A$, $b$, $x_0$, $m$, $\tau$}  \label{algo:PIPECG}
  \begin{algorithmic}[1]
    \State $r_0 := b - Ax_0$;
		\State $w_0 := Ar_0$;
    \For{$i = 0,\dots,m$}
    \State $\gamma_i := (r_i,r_i)$;
    \State $\delta_i := (w_i,r_i)$;
    \State $v_i := A w_i$;
		\If{$\sqrt{\gamma_i}/\|r_0\| < \tau$}
		\State RETURN;
		\EndIf
		\State \textbf{end if} 
    \If{$i>0$}
    \State $\beta_i := \gamma_i/\gamma_{i-1}$; 
		\State $\alpha_i := (\delta_i/\gamma_i - \beta_i/\alpha_{i-1})^{-1}$;
    \Else
    \State $\beta_i := 0$; 
		\State $\alpha_i := \gamma_i/\delta_i$;
    \EndIf
		\State \textbf{end if} 
    \State $z_i := v_i + \beta_i z_{i-1}$;
    \State $s_i := w_i + \beta_i s_{i-1}$;
    \State $p_i := r_i + \beta_i p_{i-1}$;
    \State $x_{i+1} := x_i + \alpha_i p_i$;
    \State $r_{i+1} := r_i - \alpha_i s_i$;
    \State $w_{i+1} := w_i - \alpha_i z_i$;
    \EndFor
		\State \textbf{end for} 
  \end{algorithmic}
}
\end{algorithm}

The behavior of local rounding errors in the length-one p-CG method from \cite{ghysels2014hiding}, see Algorithm \ref{algo:PIPECG}, was analyzed in our previous work \cite{cools2018analyzing} and in the related work \cite{carson2016numerical}. 
Pipelined p-CG uses additional recurrence relations for the auxiliary vector quantities $w_j := A r_j$, $s_j := A p_j$ and $z_j := A s_j$. The coupling between these recursively defined variables may cause amplification of the local rounding errors in the algorithm. In finite precision one has the following recurrence relations in p-CG (without preconditioner)\footnote{Preconditioning of the various Conjugate Gradients variants has been largely omitted in this section for simplicity of notation. The extension of the local rounding error analysis to the preconditioned p-CG (and, by extension, p($\ell$)-CG) algorithm is trivial, since the recurrences for the unpreconditioned variables are decoupled from their preconditioned counterparts. We refer the reader to Section \ref{sec:preconditioned} and our own related work in \cite{cools2018analyzing} for more details.}:
\begin{align}
	\bar{x}_{j+1} &= \bar{x}_j + \bar{\alpha}_j \bar{p}_j + \xi_{j+1}^{\bar{x}}, &
	\bar{w}_{j+1} &= \bar{w}_j - \bar{\alpha}_j \bar{z}_j + \xi_{j+1}^{\bar{w}}, &
	\bar{r}_{j+1} &= \bar{r}_j - \bar{\alpha}_j \bar{s}_j + \xi_{j+1}^{\bar{r}}, \notag \\
	\bar{s}_j 		&= \bar{w}_j + \bar{\beta}_j \bar{s}_{j-1} + \xi_j^{\bar{s}}, &
	\bar{p}_j 		&= \bar{r}_j + \bar{\beta}_j \bar{p}_{j-1} + \xi_j^{\bar{p}} , &
	\bar{z}_j     &= A\bar{w}_j + \bar{\beta}_j \bar{z}_{j-1} + \xi_j^{\bar{z}}. \label{eq:pxr_pipecg}
\end{align}
The respective bounds for the local rounding errors $\xi_{k}^{\bar{x}},\xi_{k}^{\bar{r}},\xi_{k}^{\bar{p}},\xi_{k}^{\bar{s}},\xi_{k}^{\bar{w}}$ and $\xi_{k}^{\bar{z}}$ (based on the recurrence relations above) can be found in \cite{cools2018analyzing}, Section 2.3, where it is furthermore shown that the residual gap $f_j = (b-A\bar{x}_j) - \bar{r}_j$ is coupled to the gaps $g_j = A\bar{p}_j - \bar{s}_j$, $h_j = A\bar{r}_j - \bar{w}_j$ and $e_j = A\bar{s}_j - \bar{z}_j$ on the auxiliary variables in p-CG.
We present the relations from \cite{cools2018analyzing} in matrix form. Let $B = [b,b,\ldots,b]$, $\bar{X}_{j+1} = [\bar{x}_0,\bar{x}_1,\ldots, \bar{x}_j]$ and $\bar{P}_{j+1} = [\bar{p}_0,\bar{p}_1,\ldots, \bar{p}_j]$. Writing 
the gaps 
as
\begin{align}
  \mathcal{F}_{j+1} &= R_{j+1}-\bar{R}_{j+1}, & R_{j+1} &= B-A\bar{X}_{j+1}, & \bar{R}_{j+1} &= [\bar{r}_0,\bar{r}_1,\ldots, \bar{r}_j], \notag \\
  \mathcal{G}_{j+1} &= S_{j+1}-\bar{S}_{j+1}, & S_{j+1} &= A\bar{P}_{j+1}, & \bar{S}_{j+1} &= [\bar{s}_0,\bar{s}_1,\ldots, \bar{s}_j], &  \notag \\
  \mathcal{H}_{j+1} &= W_{j+1}-\bar{W}_{j+1}, & W_{j+1} &= A\bar{R}_{j+1}, &  \bar{W}_{j+1} &= [\bar{w}_0,\bar{w}_1,\ldots, \bar{w}_j], &  \notag \\
  \mathcal{E}_{j+1} &= Z_{j+1}-\bar{Z}_{j+1}, & Z_{j+1} &= A\bar{S}_{j+1}, &  \bar{Z}_{j+1} &= [\bar{z}_0,\bar{z}_1,\ldots, \bar{z}_j], &  \notag
\end{align}
and using the following expressions for the local rounding errors on the auxiliary variables 
\begin{align}
 \Theta_j^{\bar{x}} &= -[0,\xi_1^{\bar{x}},...\,, \xi_{j-1}^{\bar{x}}], \hspace{-0.2cm} & \Theta_j^{\bar{r}} &= [f_0,-\xi_1^{\bar{r}},...\,, -\xi_{j-1}^{\bar{r}}], \hspace{-0.2cm} &
 \Theta_j^{\bar{p}} &= [0,\xi_1^{\bar{p}},...\,, \xi_{j-1}^{\bar{p}}],  \hspace{-0.2cm} & \Theta_j^{\bar{s}} &= [g_0,-\xi_1^{\bar{s}},...\,, -\xi_{j-1}^{\bar{s}}], \notag \\
 \Theta_j^{\bar{u}} &= [0,\xi_1^{\bar{r}},...\,, \xi_{j-1}^{\bar{r}}],  \hspace{-0.2cm} & \Theta_j^{\bar{w}} &= [h_0,-\xi_1^{\bar{w}},...\,, -\xi_{j-1}^{\bar{w}}], \hspace{-0.2cm}&
 \Theta_j^{\bar{q}} &= [0,\xi_1^{\bar{s}},...\,, \xi_{j-1}^{\bar{s}}],  \hspace{-0.2cm} & \Theta_j^{\bar{z}} &= [e_0,-\xi_1^{\bar{z}},...\,, -\xi_{j-1}^{\bar{z}}], \notag 
\end{align}
we obtain matrix expressions for the local rounding errors in p-CG: 
\begin{align*} 
  \mathcal{F}_{j+1} &= (A\Theta_{j+1}^{\bar{x}}+\Theta_{j+1}^{\bar{r}}) \, U_{j+1} + \mathcal{G}_{j+1} \mathcal{A}_{j+1}, &
  \mathcal{G}_{j+1} &= (A\Theta_{j+1}^{\bar{p}}+\Theta_{j+1}^{\bar{s}}) \, \mathcal{B}^{-1}_{j+1} + \mathcal{H}_{j+1} \widetilde{\mathcal{B}^{-1}_{j+1}}, \\ 
  \mathcal{H}_{j+1} &= (A\Theta_{j+1}^{\bar{u}}+\Theta_{j+1}^{\bar{w}}) \, U_{j+1} + \mathcal{E}_{j+1} \mathcal{A}_{j+1}, &
  \mathcal{E}_{j+1} &= (A\Theta_{j+1}^{\bar{q}}+\Theta_{j+1}^{\bar{z}}) \, \mathcal{B}^{-1}_{j+1},
\end{align*}
where \vspace{-0.2cm}
\begin{equation*}
  \mathcal{A}_{j+1} = -
  \left(\begin{array}{ccccc} 
    0&\bar{\alpha}_0&\bar{\alpha}_0&\cdots&\bar{\alpha}_0 \\
    &0&\bar{\alpha}_1&\cdots&\bar{\alpha}_1 \\
    &&\ddots&& \vdots\\
    &&&0& \bar{\alpha}_{j-1}\\
    &&&&0
  \end{array}\right),
\qquad
  \mathcal{B}_{j+1} = 
  \left(\begin{array}{ccccc} 
    1&-\bar{\beta}_1&0&\cdots&0 \\
    &1&-\bar{\beta}_2& \ddots &\vdots\\
    &&\ddots&\ddots&0 \\
    &&&\ddots&-\bar{\beta}_j \\
    &&&&1
  \end{array}\right).
\end{equation*}
The inverse of the upper bidiagonal matrix $\mathcal{B}_{j+1}$ can be expressed in terms of the products of the coefficients $\bar{\beta}_k$ (with $1 \leq k \leq j$), i.e.
\begin{equation*}
\mathcal{B}_{j+1}^{-1} = 
\left(\begin{array}{ccccc} 
1&\bar{\beta}_1&\bar{\beta}_1 \bar{\beta}_2&\cdots&\bar{\beta}_1 \bar{\beta}_2 \ldots \bar{\beta}_j \\
&1&\bar{\beta}_2& \ddots &\bar{\beta}_2 \ldots \bar{\beta}_j\\
&&\ddots&\ddots&\vdots \\
&&&\ddots&\bar{\beta}_j \\
&&&&1
\end{array}\right).
\end{equation*}
Furthermore, the matrix $\widetilde{\mathcal{B}_{j+1}^{-1}}$ is defined by setting the first row of $\mathcal{B}_{j+1}^{-1}$ to zero. For a correct interpretation we note that this matrix is not invertible; the notation $\widetilde{\mathcal{B}_{j+1}^{-1}}$ merely indicates the close relation 
to $\mathcal{B}_{j+1}^{-1}$. 
The residual gap in p-CG is summarized by the following expression:
\begin{align}\label{eq:total_gap_pcg}
  \mathcal{F}_{j+1} &= (A\Theta_{j+1}^{\bar{x}}+\Theta_{j+1}^{\bar{r}}) \, U_{j+1} + (A\Theta_{j+1}^{\bar{p}}+\Theta_{j+1}^{\bar{s}}) \, \mathcal{B}^{-1}_{j+1} \mathcal{A}_{j+1} + \ldots \notag \\
	& \quad + (A\Theta_{j+1}^{\bar{u}}+\Theta_{j+1}^{\bar{w}}) \, U_{j+1} \widetilde{\mathcal{B}^{-1}_{j+1}} \mathcal{A}_{j+1} + (A\Theta_{j+1}^{\bar{q}}+\Theta_{j+1}^{\bar{z}}) \, \mathcal{B}^{-1}_{j+1} \mathcal{A}_{j+1} \widetilde{\mathcal{B}^{-1}_{j+1}} \mathcal{A}_{j+1}, 
\end{align}
Hence, the entries of the coefficient matrices $\mathcal{B}^{-1}_{j+1}$ and $\mathcal{A}_{j+1}$ determine the propagation of the local rounding errors in p-CG.
The entries of $\mathcal{B}^{-1}_{j+1}$ consist of a product of the scalar coefficients $\bar{\beta_{j}}$. In exact arithmetic these coefficients equal $\beta_j = \|r_j\|^2/\|r_{j-1}\|^2$, such that
\begin{equation}
	\beta_i \, \beta_{i+1} \, \ldots \, \beta_{j} = \frac{\|r_i\|^2}{\|r_{i-1}\|^2} \frac{\|r_{i+1}\|^2}{\|r_{i}\|^2} \ldots \frac{\|r_j\|^2}{\|r_{j-1}\|^2} = \frac{\|r_j\|^2}{\|r_{i-1}\|^2}, \quad i \leq j.
\end{equation}
Since the residual norm in CG is not guaranteed to decrease monotonically, the factor $\|r_j\|^2/\|r_{i-1}\|^2$ may for some $i \leq j$ be much larger than one. A similar argument may be used in the finite precision framework to derive that some entries of $\mathcal{B}^{-1}_{j+1}$ may be significantly larger than one, and may hence (possibly dramatically) amplify the corresponding local rounding errors in expression \eqref{eq:total_gap_pcg}. This behavior is illustrated in Section \ref{sec:numerical}, Fig.~\ref{fig:figure2} (top right), where the p-CG residual stagnates at a reduced maximal attainable accuracy level compared to classic CG, see Fig.~\ref{fig:figure2} (top left).

\section{Analyzing local rounding error propagation in p($\ell$)-CG} \label{sec:pipelcg}

This section aims to analyze the behavior of local rounding errors that stem from the recurrence relations in the pipelined p($\ell$)-CG algorithm with deep pipelines, Algorithm \ref{algo:PIPELCG}. The analysis will follow a framework similar to Section \ref{sec:pipecg}; however, the p($\ell$)-CG method is in fact much more closely related to p($\ell$)-GMRES \cite{ghysels2013hiding} than it is to p-CG \cite{ghysels2014hiding}. For a proper understanding we first resume the essential definitions and recurrence relations used in the p($\ell$)-CG algorithm.

\subsection{Basis vector recurrence relations in p($\ell$)-CG in exact arithmetic}

\begin{algorithm}[t]
{\small
\caption{Deep pipelined CG (p($\ell$)-CG) \hfill \textbf{Input:} $A$, $b$, $x_0$, $l$, $m$, $\sigma_0,\ldots,\sigma_{l-1}$, $\tau$}\label{algo:PIPELCG}
\begin{algorithmic}[1]
\State $r_{0}:=b-Ax_{0};$ $v_{0}:= r_{0}/{\|r_{0}\|}_2;$ $z_{0}:=v_{0}; ~ g_{0,0}:=1;$
\For {$i=0,\ldots, m+l$}
\State $ z_{i+1}:=\left\{ \begin{matrix}(A-\sigma_{i}I)z_{i}, & i<l \\ Az_{i}, & i \geq l \end{matrix}\right.$
\State $a:=i-l$
\If {$a\geq 0$}
\State $g_{j,a+1} := (g_{j,a+1}-\sum_{k=a+1-2l}^{j-1}g_{k,j}g_{k,a+1})/g_{j,j}; \qquad j=a-l+2,\ldots,a$  
\State $g_{a+1,a+1}:= \sqrt{g_{a+1,a+1}-\sum_{k=a+1-2l}^{a}g_{k,a+1}^2};$
\State \# Check for breakdown and restart if required
\If {$a<l$}
\State $\gamma_{a}:=(g_{a,a+1}+\sigma_{a}g_{a,a}-\delta_{a-1}g_{a-1,a})/g_{a,a};$
\State $\delta_{a}:=g_{a+1,a+1}/g_{a,a};$
\Else
\State $\gamma_{a}:=(g_{a,a}\gamma_{a-l}+g_{a,a+1}\delta_{a-l}-\delta_{a-1}g_{a-1,a})/g_{a,a};$
\State $\delta_{a}:=(g_{a+1,a+1}\delta_{a-l})/g_{a,a};$
\EndIf
\State \textbf{end if}
\State $v_{a+1} := (z_{a+1} - \sum_{j=a-2l+1}^{a} g_{j,a+1} v_{j})/g_{a+1,a+1};$
\State $z_{i+1} := (z_{i+1} - \gamma_{a} z_{i}- \delta_{a-1} z_{i-1})/\delta_{a};$
\EndIf
\State \textbf{end if}
\If {$a<0$}
\State $g_{j,i+1}:=(z_{i+1},z_{j}); \qquad j=0,\ldots,i+1$
\Else 
\State $g_{j,i+1}:=\left\{ \begin{matrix}(z_{i+1},v_{j}); & j=\max(0,i-2l+1),\ldots,a+1 \\ (z_{i+1},z_{j});  &j=a+2,\ldots,i+1 \end{matrix}\right.$
\EndIf
\State \textbf{end if}
\If {$a=0$}
\State $\eta_{0}:=\gamma_{0};$
\State $\zeta_{0}:={\|r_{0}\|}_2;$
\State $p_{0}:=v_0/\eta_0;$
\Else \, \textbf{if} {$a\geq 1$} \textbf{then}
\State $\lambda_{a}:=\delta_{a-1}/\eta_{a-1};$ 
\State $\eta_{a}:=\gamma_{a}-\lambda_{a}\delta_{a-1};$
\State $\zeta_{a}=-\lambda_{a}\zeta_{a-1};$
\State $p_{a}=(v_{a}-\delta_{a-1}p_{a-1})/\eta_{a};$
\State $x_{a}=x_{a-1}+\zeta_{a-1}p_{a-1};$
\If {$|\zeta_a|/\|r_0\| < \tau$}
\State RETURN;
\EndIf
\State \textbf{end if} 
\EndIf
\State \textbf{end if}
\EndFor 
\State \textbf{end for}
\end{algorithmic}
}
\end{algorithm}

Let $V_{i-l+1}:=[v_{0},v_{1},\ldots,v_{i-l}]$ be the orthonormal basis for the Krylov subspace $\mathcal{K}_{i-l+1}(A,v_{0})$ in iteration $i$ of the p($\ell$)-CG algorithm, where $A$ is a symmetric matrix. These vectors satisfy the Arnoldi relation
\begin{equation} 
AV_{j} = V_{j+1} H_{j+1,j}, \qquad 1 \leq j \leq i-l, 
\end{equation}
where $H_{j+1,j}$ is the $(j+1) \times j$ tridiagonal matrix
\begin{equation*} H_{j+1,j} = 
  \begin{pmatrix} 
  \gamma_{0} & \delta_{0} & & & \\ 
 \delta_{0} & \gamma_{1} & \delta_{1} & & \\ 
  & \delta_{1} & \gamma_{2} & \ddots & \\ 
  & & \ddots & \ddots & \delta_{j-2} \\ 
  & & & \delta_{j-2} & \gamma_{j-1}  \\
  & & & & \delta_{j-1}
\end{pmatrix}.
\end{equation*}
For  $1 \leq j \leq i-l$ the relation $AV_{j} = V_{j+1} H_{j+1,j}$ translates in vector notation to the recursive definition of $v_{j+1}$:
\begin{equation}\label{eq:v_ex}
  v_{j+1}= ( Av_{j} - \gamma_j v_j - \delta_{j-1} v_{j-1} ) / \delta_j, \qquad 0 \leq j < i - l.
\end{equation}
Note that for $j = 0$ it is assumed that $v_{-1} = 0$. We define the auxiliary vector basis $Z_{i+1}:=[z_{0},z_{1},\ldots,z_{i-l},z_{i-l+1},\ldots,z_{i}]$, which runs $l$ vectors ahead of the basis $V_{i-l+1}$ (i.e.~the so-called \emph{pipeline} of length $l$) as 
\begin{equation} \label{eq:z_ex}
  z_{j}:= \left\{ \begin{matrix} v_{0}, & j=0, \\ P_{j}(A)v_{0}, & 0<j\leq l, \\ P_{l}(A)v_{j-l}, & j>l, \end{matrix} \right. \qquad \text{with} \qquad P_{i}(t) := \prod_{j=0}^{i-1} (t-\sigma_{j}), \qquad i\leq l,
\end{equation}
with optional stabilizing shifts $\sigma_{j}\in \mathbb{C}$, see \cite{ghysels2013hiding,cornelis2017communication}. 
We comment on the choice of the Krylov basis and its effect on numerical stability further on in this manuscript.
Note that contrary to the basis $V_{i-l+1}$, the auxiliary basis $Z_{i+1}$ is in general not orthonormal.
For $j < l$ one has the relation:
\begin{equation} \label{eq:z_rec0}
z_{j+1} = (A-\sigma_{j}I) \, z_{j}, \qquad 0 \leq j < l,
\end{equation}
whereas for $j \geq l$ the relation \eqref{eq:v_ex} for $v_{j-l+1}$ is multiplied on both sides by $P_{l}(A)$ to obtain the recurrence relation for $z_{j+1}$:
\begin{equation} \label{eq:z_rec}
  z_{j+1} = ( Az_{j} - \gamma_{j-l} z_j - \delta_{j-l-1} z_{j-1} ) / \delta_{j-l}, \qquad l \leq j < i. 
\end{equation}
In matrix formulation the expressions \eqref{eq:z_rec0}-\eqref{eq:z_rec} for $z_{j+1}$ translate into the Arnoldi-like relation
\begin{equation}
A Z_j = Z_{j+1} B_{j+1,j}, \qquad 1 \leq j \leq i,
\end{equation}
 where the matrix $B_{j+1,j}$ is
\begin{equation*}
B_{j+1,j}
= 
\left(\begin{array}{ccc|ccc} 
\sigma_{0} & & & & & \\
1 & \ddots & & & & \\
& \ddots & \sigma_{l-1} & & & \\ \hline
& & 1 & & & \\
& & & & H_{j-l+1,j-l} & \\
& & & & & 
\end{array}\right).
\end{equation*}

The basis vectors $V_{j}$ and $Z_j$ are connected through the basis transformation $Z_{j}=V_{j}G_{j}$ for $1 \leq j \leq i-l+1$. The upper triangular matrix $G_{j}$ has a band structure with a band width of at most $2l+1$ non-zero diagonals, see \cite{cornelis2017communication}, Lemma 6. Using this basis transformation the following recurrence relation for $v_{j+1}$ is derived:
\begin{equation} \label{eq:v_rec}
  v_{j+1} = \left( z_{j+1} - \sum_{k=j-2l+1}^{j} g_{k,j+1} v_{k} \right) / g_{j+1,j+1}, \qquad 0 \leq j < i-l.
\end{equation}
The recurrence relations \eqref{eq:z_rec} and \eqref{eq:v_rec} are used in Alg.~\ref{algo:PIPELCG} to recursively compute the respective basis vectors $v_{i-l+1}$ and $z_{i+1}$ in iterations $i \geq l$ (i.e. once the initial pipeline for $z_0,\ldots,z_l$ has been filled).

\begin{remark} \label{remark:residualnorm}
  Note that although the residual $r_j = b-Ax_j$ is not recursively computed in Algorithm \ref{algo:PIPELCG}, it is proven in \cite{cornelis2017communication} that the residual norm can be characterized by the quantity $\zeta_j$, i.e.
	\begin{equation} \label{eq:resnorm}
		\|r_j\| = |\zeta_j|, \qquad 0 \leq j \leq i-l.
	\end{equation}
\end{remark}

\begin{remark} \label{remark:sqrt_breakdown}
Unlike the p-CG Algorithm \ref{algo:PIPECG}, the p($\ell$)-CG Algorithm \ref{algo:PIPELCG} may encounter a square root breakdown during the iteration. When the root argument $g_{a+1,a+1}-\sum_{k=a+1-2l}^{a}g_{k,a+1}^2$ becomes negative (possibly influenced by round-off errors in practice) a breakdown occurs. It is suggested in \cite{cornelis2017communication} that when breakdown occurs the iteration is restarted in analogy to the (pipelined) GMRES algorithm. Note however that this restarting strategy may delay convergence, cf.~Section \ref{sec:numerical}, Fig.~\ref{fig:figure1}.
\end{remark}

\subsection{Local rounding error behavior in finite precision p($\ell$)-CG} \label{sec:local}

For any $j \geq 0$ the true basis vector, denoted by $\bar{\bold{v}}_{j+1}$, satisfies the Arnoldi relation \eqref{eq:v_ex} exactly, that is, it is defined as
\begin{equation} \label{eq:v_arnoldi}
  \bar{\bold{v}}_{j+1}= ( A\bar{v}_{j} - \bar{\gamma}_j \bar{v}_j - \bar{\delta}_{j-1} \bar{v}_{j-1} ) / \bar{\delta}_j, \qquad 0 \leq j < i-l.
\end{equation}
For $j = 0$ it is assumed that $\bar{v}_{-1} = 0.$
On the other hand, the computed basis vector $\bar{v}_{j+1}$ is calculated from the finite precision variant of the recurrence relation \eqref{eq:v_rec}, i.e.
\begin{equation} \label{eq:vbar_rec}
  \bar{v}_{j+1} = \left( \bar{z}_{j+1} - \sum_{k=j-2l+1}^{j} \bar{g}_{k,j+1} \bar{v}_{k} \right) / \bar{g}_{j+1,j+1} + \xi^{\bar{v}}_{j+1}, \qquad 0 \leq j < i-l,
\end{equation}
where the size of the local rounding errors $\xi^{\bar{v}}_{j+1}$ can be bounded in terms of the machine precision $\epsilon$ as follows:
\begin{equation*} 
  \| \xi^{\bar{v}}_{j+1} \| \leq \left( 2 \, \frac{1}{|\bar{g}_{j+1,j+1}|} \|\bar{z}_{j+1}\| + 3 \sum_{k=j-2l+1}^{j} \frac{|\bar{g}_{k,j+1}|}{|\bar{g}_{j+1,j+1}|} \, \|\bar{v}_k\| \right) \epsilon.
\end{equation*}

By subtracting the computed basis vector $\bar{v}_{j+1}$ from both sides of the equation \eqref{eq:v_arnoldi}, it is easy to see that this relation alternatively translates to
\begin{equation*}
  A\bar{v}_{j} = \bar{\delta}_{j-1} \bar{v}_{j-1} + \bar{\gamma}_j \bar{v}_j +  \bar{\delta}_j \bar{v}_{j+1} + \bar{\delta}_j ( \bar{\bold{v}}_{j+1} - \bar{v}_{j+1} ) , \qquad 0 \leq j < i-l
\end{equation*}
or written in matrix notation:
\begin{equation} \label{eq:AV_BAR}
  A \bar{V}_j = \bar{V}_{j+1} \bar{H}_{j+1,j} + (\bar{\boldsymbol{V}}_{j+1}-\bar{V}_{j+1}) \bar{\Delta}_{j+1,j}, \qquad 1 \leq j \leq i-l,
\end{equation}
where $\bar{\Delta}_{j+1,j}$ is
\begin{equation*}
\bar{\Delta}_{j+1,j} =
\left(\begin{array}{cccc} 
0&&& \\
\bar{\delta}_0&0&& \\
&\bar{\delta}_1&0& \\
&&\ddots&0 \\
&&&\bar{\delta}_{j-1}
\end{array}\right).
\end{equation*}
Note that the Moore-Penrose (left) pseudo-inverse $\bar{\Delta}_{j+1,j}^{+} = (\bar{\Delta}_{j+1,j}^* \bar{\Delta}_{j+1,j})^{-1} \bar{\Delta}_{j+1,j}^*$ of this lower diagonal matrix is an upper diagonal matrix, where $\bar{\Delta}_{j+1,j}^*$ is the Hermitian transpose of $\bar{\Delta}_{j+1,j}$.

By setting $\Theta_j^{\bar{v}} = [\theta^{\bar{v}}_0,\theta^{\bar{v}}_1,\ldots,\theta^{\bar{v}}_{j-1}] := [\bar{g}_{0,0}\xi^{\bar{v}}_0, \, \bar{g}_{1,1}\xi^{\bar{v}}_1, \, \ldots, \, \bar{g}_{j-1,j-1}\xi^{\bar{v}}_{j-1}]$, we obtain from \eqref{eq:vbar_rec} the matrix expression
\begin{equation} \label{eq:Z_BAR}
  \bar{Z}_j = \bar{V}_j \bar{G}_j + \Theta^{\bar{v}}_j, \qquad 1 \leq j \leq i-l.
\end{equation}

For $j \geq l$, the computed auxiliary vector $\bar{z}_{j+1}$ satisfies a finite precision version of the recurrence relation \eqref{eq:z_rec}, which pours down to
\begin{equation} \label{eq:zbar_rec}
  \bar{z}_{j+1} = ( A\bar{z}_{j} - \bar{\gamma}_{j-l} \bar{z}_j - \bar{\delta}_{j-l-1} \bar{z}_{j-1} ) / \bar{\delta}_{j-l} + \xi^{\bar{z}}_{j+1}, \qquad l \leq j < i,
\end{equation}
where the local rounding error $\xi^{\bar{z}}_{j+1}$ for $j \geq l$ is bounded by
\begin{equation*}
  \| \xi^{\bar{z}}_{j+1} \| \leq \left( (\mu \sqrt{n}+2) \frac{\|A\|}{|\bar{\delta}_{j-l}|} \|\bar{z}_j\| + 3 \frac{|\bar{\gamma}_{j-l}|}{|\bar{\delta}_{j-l}|} \|\bar{z}_j\| + 3 \frac{|\bar{\delta}_{j-l-1}|}{|\bar{\delta}_{j-l}|} \|\bar{z}_{j-1}\| \right) \epsilon.
\end{equation*}
Here $n$ is the number of rows/columns in the matrix $A$, and $\mu$ is the maximum number of non-zeros over all rows of $A$. In the first $l$ iterations of the algorithm, i.e.~for $0 \leq j < l$, the next auxiliary basis vector $\bar{z}_{j+1}$ is not yet computed recursively, but is instead computed directly by applying the matrix $(A-\sigma_jI)$ to the previous basis vector $\bar{z}_j$, i.e.~in finite precision
\begin{equation} \label{eq:zbar_rec_b}
  \bar{z}_{j+1} = (A-\sigma_jI) \, \bar{z}_j + \xi^{\bar{z}}_{j+1}, \qquad 0 \leq j < l.
\end{equation} 
This implies the local rounding error can be bounded by $\| \xi^{\bar{z}}_{j+1} \| \leq \mu \sqrt{n} \, \|A-\sigma_j I\| \, \|\bar{z}_j\| \, \epsilon$ when $j < l$.
Expressions \eqref{eq:zbar_rec} and \eqref{eq:zbar_rec_b} can be summarized in matrix notation as:
\begin{equation} \label{eq:AZ_BAR}
  A \bar{Z}_j = \bar{Z}_{j+1} \bar{B}_{j+1,j} + \Theta^{\bar{z}}_j, \qquad 1 \leq j \leq i,
\end{equation}
where $\Theta^{\bar{z}}_j = [\theta^{\bar{z}}_0,\theta^{\bar{z}}_1,\ldots,\theta^{\bar{z}}_{j-1}]$, with $\theta^{\bar{z}}_{k} = \xi^{\bar{z}}_{k+1}$ for $0 \leq k < l$, and $\theta^{\bar{z}}_{k} = \bar{\delta}_{k-l}\xi^{\bar{z}}_{k+1}$ for $l \leq k < i$.

Furthermore, the recursive definitions of the scalar coefficients $\bar{\gamma}_j$ and $\bar{\delta}_j$ in Alg.~\ref{algo:PIPELCG} imply that in iteration $i$ the following matrix relations hold:
\begin{equation} \label{eq:rec_coeff}
  \bar{G}_{j+1}  \bar{B}_{j+1,j} = \bar{H}_{j+1,j} \bar{G}_j, \qquad 1 \leq j \leq i-l.
\end{equation}

The (scaled) vector $\bar{v}_{i-l}$ is used in p($\ell$)-CG to update the search direction $\bar{p}_{i-l}$. Hence we aim to derive an expression for the gap between the true basis vector $\bar{\bold{v}}_{j+1}$ and the corresponding computed vector $\bar{v}_{j+1}$ in each iteration of the p($\ell$)-CG algorithm. We denote this gap by re-using the notation $f_{j}$ in this setting, i.e.~$f_{j+1} := \bar{\bold{v}}_{j+1} - \bar{v}_{j+1}$. In matrix notation we define:
\begin{equation}
 \mathcal{F}_{j+1} := \bar{\boldsymbol{V}}_{j+1} - \bar{V}_{j+1}.
\end{equation}


In analogy to Section \ref{sec:rounding_pcg} we will write the gap $\mathcal{F}_{j+1}$ for the vectors $v_0,\ldots,v_j$ in terms of the gap for the auxiliary basis vectors $z_0,\ldots,z_j$. We define
\begin{equation}
 \mathcal{G}_{j+1} := \bar{\boldsymbol{Z}}_{j+1} - \bar{Z}_{j+1}.
\end{equation}
where the true basis vectors $\bar{\boldsymbol{Z}}_{j+1}$ are defined as in \eqref{eq:z_ex}, i.e.:
\begin{equation} \label{eq:z_rec3}
  \bar{\bold{z}}_{j}:= \left\{ \begin{matrix} \bar{\bold{v}}_{0}, & j=0, \\ P_{j}(A)\bar{\bold{v}}_{0}, & 0<j\leq l, \\ P_{l}(A)\bar{\bold{v}}_{j-l}, & j>l, \end{matrix} \right.
\end{equation}
It readily follows from this definition and the relation \eqref{eq:v_arnoldi} that 
\begin{equation}\label{eq:z_rec2}
	\bar{\bold{z}}_{j+1} = ( A \bar{z}_j - \bar{\gamma}_{j-l} \bar{z}_{j} - \bar{\delta}_{j-l-1} \bar{z}_{j-l-1} ) / \bar{\delta}_{j-l}, \qquad l \leq j < i,
\end{equation}
which is the analogue to expression \eqref{eq:zbar_rec} without the local rounding error term.
 
The gaps $\mathcal{F}_{j+1}$ and $\mathcal{G}_{j+1}$ in p($\ell$)-CG can be computed as follows. 
By subsequently subtracting $\bar{z}_{j+1}$ from both sides of \eqref{eq:z_rec2}, one obtains the following relation in matrix notation:
\begin{equation} \label{eq:AZ_BAR2}
	A\bar{Z}_{j} = \bar{Z}_{j+1} \bar{B}_{j+1,j} + (\bar{\boldsymbol{Z}}_{j+1}-\bar{Z}_{j+1}) \bar{\nabla}_{j+1,j}, \qquad 1 \leq j \leq i,
\end{equation}
where $\bar{\nabla}_{j+1,j}$ is
\begin{equation*}
\bar{\nabla}_{j+1,j} =
\left(\begin{array}{cccccc} 
0&&&&& \\
1&\ddots&&&& \\
&\ddots&0&&& \\
&&1&0&& \\
&&&\bar{\delta}_0&\ddots& \\
&&&&\ddots&0 \\
&&&&&\bar{\delta}_{j-l-1}
\end{array}\right).
\end{equation*}
It immediately follows by comparing \eqref{eq:AZ_BAR} to \eqref{eq:AZ_BAR2} that the gap $\mathcal{G}_{j+1}$ can be expressed as
\begin{equation} \label{eq:gap_plcg3}
\mathcal{G}_{j+1} = \bar{\boldsymbol{Z}}_{j+1} - \bar{Z}_{j+1} =  \Theta^{\bar{z}}_j \, \bar{\nabla}_{j+1,j}^{+}.
\end{equation}
Here $\bar{\nabla}_{j+1,j}^{+}$ should again be interpreted as a pseudo-inverse, i.e.~$\bar{\nabla}_{j+1,j}^{+} = (\bar{\nabla}_{j+1,j}^* \bar{\nabla}_{j+1,j})^{-1} \bar{\nabla}_{j+1,j}^*$.
The gap $\mathcal{F}_{j+1}$ can then be computed in terms of the gap $\mathcal{G}_{j+1}$. It holds that
\begin{align} \label{eq:gap_plcg2}
\mathcal{F}_{j+1} &= \bar{\boldsymbol{V}}_{j+1} - \bar{V}_{j+1} \notag \\
				&\stackrel{\eqref{eq:AV_BAR}}{=}  (A \bar{V}_j - \bar{V}_{j+1} \bar{H}_{j+1,j}) \bar{\Delta}^{+}_{j+1,j} \notag \\
				&\stackrel{\eqref{eq:Z_BAR}}{=} (A \bar{Z}_j \bar{G}^{-1}_j - \bar{Z}_{j+1} \bar{G}^{-1}_{j+1} \bar{H}_{j+1,j} - A \Theta^{\bar{v}}_j \bar{G}^{-1}_j  + \Theta^{\bar{v}}_{j+1} \bar{G}^{-1}_{j+1} \bar{H}_{j+1,j}) \bar{\Delta}^{+}_{j+1,j} \notag \\
				&\stackrel{\eqref{eq:rec_coeff}}{=} (A \bar{Z}_j \bar{G}^{-1}_j - \bar{Z}_{j+1} \bar{B}_{j+1,j} \bar{G}^{-1}_j - A \Theta^{\bar{v}}_j \bar{G}^{-1}_j  + \Theta^{\bar{v}}_{j+1} \bar{G}^{-1}_{j+1} \bar{H}_{j+1,j}) \bar{\Delta}^{+}_{j+1,j} \notag \\
				&\stackrel{\eqref{eq:AZ_BAR2}}{=} (\mathcal{G}_{j+1}  \bar{\nabla}_{j+1,j}\bar{G}^{-1}_j - A \Theta^{\bar{v}}_j \bar{G}^{-1}_j  + \Theta^{\bar{v}}_{j+1} \bar{G}^{-1}_{j+1} \bar{H}_{j+1,j}) \bar{\Delta}^{+}_{j+1,j}.
\end{align}
Alternatively, the formula for the gap $\mathcal{F}_{j+1} = \bar{\boldsymbol{V}}_{j+1} - \bar{V}_{j+1}$ in the p($\ell$)-CG algorithm can be derived directly by using expression \eqref{eq:AZ_BAR} as follows:
\begin{align} \label{eq:gap_plcg}
\mathcal{F}_{j+1} 
				&\stackrel{\eqref{eq:Z_BAR}}{=}  (A \bar{Z}_j \bar{G}^{-1}_j - \bar{Z}_{j+1} \bar{G}^{-1}_{j+1} \bar{H}_{j+1,j} - A \Theta^{\bar{v}}_j \bar{G}^{-1}_j  + \Theta^{\bar{v}}_{j+1} \bar{G}^{-1}_{j+1} \bar{H}_{j+1,j}) \bar{\Delta}^{+}_{j+1,j} \notag \\
				&\stackrel{\eqref{eq:AZ_BAR}}{=}  (\bar{Z}_{j+1} \bar{B}_{j+1,j} \bar{G}^{-1}_j - \bar{Z}_{j+1} \bar{G}^{-1}_{j+1} \bar{H}_{j+1,j} + \Theta^{\bar{z}}_j \bar{G}^{-1}_j - A \Theta^{\bar{v}}_j \bar{G}^{-1}_j  + \Theta^{\bar{v}}_{j+1} \bar{G}^{-1}_{j+1} \bar{H}_{j+1,j}) \bar{\Delta}^{+}_{j+1,j} \notag \\
				&\stackrel{\eqref{eq:rec_coeff}}{=} (\Theta^{\bar{z}}_j - A \Theta^{\bar{v}}_j + \Theta^{\bar{v}}_{j+1} \bar{B}_{j+1,j} ) \, \bar{G}^{-1}_j  \bar{\Delta}^{+}_{j+1,j},
\end{align}
which is identical to expression \eqref{eq:gap_plcg2} by performing the substitution \eqref{eq:gap_plcg3}.
Consequently, it is clear that in p($\ell$)-CG the local rounding errors in $\Theta^{\bar{z}}_j$, $A \Theta^{\bar{v}}_j$ and $\Theta^{\bar{v}}_{j+1} \bar{B}_{j+1,j}$ are possibly amplified by the entries of the matrix $\bar{G}^{-1}_j\bar{\Delta}^{+}_{j+1,j}$. 
More insights on the interpretation of expressions \eqref{eq:gap_plcg2}-\eqref{eq:gap_plcg} are provided in Section \ref{sec:further}.

\subsection{Relation to the residual gap in p($\ell$)-CG} \label{sec:residual}

It is well-known that in exact arithmetic the residual $b-Ax_j$ is closely related to the basis vector $v_j$ via the following relation, see \cite{saad2003iterative,cornelis2017communication}:
\begin{equation} \label{eq:trures}
	b-Ax_j = - \delta_{j-1} (e_j^T y_j) v_j = \|b-Ax_j\| \, v_j, \qquad \text{with~} y_j = H_{j,j}^{-1} \|r_0\| e_1,
\end{equation}
where the last equality follows from the fact that $v_j$ is normalized. Furthermore, we recall that although no residual is computed (explicitly nor recursively) in the p($\ell$)-CG algorithm, the residual norm is computed in Algorithm \ref{algo:PIPELCG} as $\|r_j\| = |\zeta_j|$, see Remark \ref{remark:residualnorm}. 

In the finite precision framework the true residual, denoted as $\bar{\bold{r}}_j$, is thus related to the p($\ell$)-CG basis vector $\bar{\bold{v}}_j$, defined by \eqref{eq:v_arnoldi}, as 
\begin{equation} \label{eq:extra_res1}
	\bar{\bold{r}}_j = b-A\bar{x}_j 
	=  \|b-A\bar{x}_j\| \, \bar{\bold{v}}_j, \qquad j \geq 0.
\end{equation}
Based on the same expression \eqref{eq:trures}, we define the implicitly computed recursive residual $\bar{r}_j$, which corresponds to the residual norm $\|\bar{r}_j\| = |\bar{\zeta}_j|$, as
\begin{equation} \label{eq:extra_res2}
	\bar{r}_j = \bar{r}_0 - \bar{V}_{j+1} \bar{H}_{j+1,j} \bar{y}_j = -\bar{\delta}_{j-1} (e_j^T \bar{y}_j) \, \bar{v}_j =  |\bar{\zeta}_j| \, \bar{v}_j, \qquad  \text{with~} \bar{y}_j = \bar{H}_{j,j}^{-1} \|\bar{r}_0\| e_1, \qquad j \geq 0,
\end{equation}
where $\bar{v}_j$ is the recursively computed basis vector defined by \eqref{eq:vbar_rec}. Hence, the 
gap between the true residual and the computed residual can be related directly to the basis vectors $\bar{\bold{v}}_j$ and $\bar{v}_j$ as follows:
\begin{align} \label{eq:resgap}
	\bar{\bold{r}}_j - \bar{r}_j &= \|b-A\bar{x}_j\| \, \bar{\bold{v}}_j - |\bar{\zeta}_j| \, \bar{v}_j \notag \\
															 &=  \|b-A\bar{x}_j\| \, (\bar{\bold{v}}_j - \bar{v}_j) + (\|b-A\bar{x}_j\| - |\bar{\zeta}_j|) \, \bar{v}_j, \qquad j \geq 0. 
\end{align}
The latter equality translates in matrix notation to
\begin{equation} \label{eq:gap_rlcg}
	\bar{\boldsymbol{R}}_{j+1} - \bar{R}_{j+1} = (\bar{\boldsymbol{V}}_{j+1} - \bar{V}_{j+1}) \, \bar{\Gamma}_{j+1} + \bar{V}_{j+1} \, \bar{\Omega}_{j+1},
\end{equation}
where $\bar{\Gamma}_{j+1}$ and $\bar{\Omega}_{j+1}$ are diagonal scaling matrices that are defined as
\begin{equation*}
  \bar{\Gamma}_{j+1} = 
  \left(\begin{array}{ccc} 
    \|b-A\bar{x}_0\|&& \\
    &\ddots& \\
    &&\|b-A\bar{x}_j\|
  \end{array}\right),
\qquad
  \bar{\Omega}_{j+1} = \bar{\Gamma}_{j+1} -  
  \left(\begin{array}{ccc} 
    |\bar{\zeta}_0|&& \\
    &\ddots& \\
    &&|\bar{\zeta}_j|
  \end{array}\right).
\end{equation*}
The intuitive expression \eqref{eq:gap_rlcg} directly links the residual gap $\bar{\bold{r}}_j - \bar{r}_j$ to the basis vector gap $\bar{\bold{v}}_{j} - \bar{v}_{j}$ characterized by expression \eqref{eq:gap_plcg}. The expression for the residual gap consists of two parts. The first term is the (possibly dramatically) increasing basis vector gap $\bar{\bold{v}}_{j} - \bar{v}_{j}$, which depends on the inverse of the basis transformation matrix $\bar{G}_j$ as established in Section \ref{sec:local}, see \eqref{eq:gap_plcg}, multiplied by the (not necessarily monotonically) decreasing true residual norm $\|b-A\bar{x}_j\|$, which 
stagnates eventually. 
The second error term consists of the vector $\bar{v}_{j}$, whose norm approximates one, multiplied by the scalar factor $\|b-A\bar{x}_j\| - |\bar{\zeta}_j|$, which becomes increasingly large relative to $\|b-A\bar{x}_j\|$ as the algorithm converges. 
This second term becomes dominant when the algorithm gets close to the stagnation point, see Fig.\,\ref{fig:figure2}
. Indeed, stagnation is reached when the norm of the residual gap $\|\bar{\bold{r}}_j - \bar{r}_j\|$ equals the norm of the true residual $\|b-A\bar{x}_j\|$, or, alternatively, when $\|b-A\bar{x}_j\| - |\bar{\zeta}_j|$ is of the same order of magnitude as $\|b-A\bar{x}_j\|$, since it holds that	$| \, \|b-A\bar{x}_j\| - \|\bar{r}_j\| \, | \leq \|(b-A\bar{x}_j) - \bar{r}_j\|$.

The analysis can be taken one step further by considering 
the finite precision variant of the recurrence relation for $\bar{x}_j$ in Algorithm \ref{algo:PIPELCG}:
\begin{equation} \label{eq:exp_x}
	\bar{x}_j = \bar{x}_{j-1} + \bar{\zeta}_{j-1} \bar{p}_{j-1} + \xi^{\bar{x}}_j 
						= \bar{x}_0 + \bar{P}_{j} \bar{L}_j^{-1} \|\bar{r}_0\| e_1 + \Theta_j^{\bar{x}} \, \bar{\boldsymbol{1}},
\end{equation}
where $\bar{L}_{j}$ is the lower triangular part in the LU factorization of the tridiagonal matrix $\bar{H}_{j,j} = \bar{L}_{j} \bar{U}_{j}$ (which is implicitly computed in Alg.\,\ref{algo:PIPELCG}), $\Theta_j^{\bar{x}} = [\xi^{\bar{x}}_0,\xi^{\bar{x}}_1,\ldots,\xi^{\bar{x}}_{j-1}]$ with $\|\xi^{\bar{x}}_j\| \leq (\|\bar{x}_{j-1}\| + 2 \, |\bar{\zeta}_{j-1}| \, \|\bar{p}_{j-1}\|) \, \epsilon$ are local rounding errors, and $\bar{\boldsymbol{1}} = [1,1,\ldots,1]^T$. 
Similarly, the finite precision recurrence relation for $\bar{p}_j$ in Alg.\,\ref{algo:PIPELCG} is
\begin{equation} \label{eq:exp_p}
	\bar{p}_j = (\bar{v}_{j} - \bar{\delta}_{j-1} \bar{p}_{j-1})/\bar{\eta}_j + \xi^{\bar{p}}_j \qquad \Leftrightarrow \qquad \bar{V}_{j+1} = \bar{P}_{j+1} \bar{U}_{j+1} + \Theta^{\bar{p}}_{j+1},
\end{equation} 
where $\bar{U}_{j}$ is the upper triangular part of $\bar{H}_{j,j} = \bar{L}_{j} \bar{U}_{j}$ and $ \Theta^{\bar{p}}_{j+1} = -[\bar{\eta}_0\xi^{\bar{p}}_0,\bar{\eta}_1\xi^{\bar{p}}_1,\ldots,\bar{\eta}_j\xi^{\bar{p}}_j]$ with $\|\xi^{\bar{p}}_j\| \leq (2 / \bar{\eta}_j \, \|\bar{v}_{j-1}\|  + 3 \, |\bar{\delta}_{j-1}|/ \bar{\eta}_j  \, \|\bar{p}_{j-1}\|) \, \epsilon$ are local rounding errors.
Substitution of expression \eqref{eq:exp_p}, i.e.~$\bar{P}_j = (\bar{V}_j - \Theta_j^{\bar{p}}) \, \bar{U}^{-1}_{j}$, 
 into equation \eqref{eq:exp_x} yields
\begin{equation} \label{eq:exp_x2}
	\bar{x}_j = \bar{x}_0 + (\bar{V}_j - \Theta_j^{\bar{p}}) \, \bar{U}^{-1}_{j} \bar{L}_j^{-1} \|\bar{r}_0\| e_1 + \Theta_j^{\bar{x}} \, \bar{\boldsymbol{1}} 
						= \bar{x}_0 + \bar{V}_j \bar{y}_j - \Theta_j^{\bar{p}} \bar{y}_j + \Theta_j^{\bar{x}} \, \bar{\boldsymbol{1}},
\end{equation}
where we used that $\bar{y}_j = \bar{H}_{j,j}^{-1} \|\bar{r}_0\| e_1$, see \eqref{eq:extra_res2}.
Consequently, the true residual $\bar{\bold{r}}_j$ can be written as
\begin{align} \label{eq:extra_res3}
	\bar{\bold{r}}_j = b-A\bar{x}_j  
								&\stackrel{\eqref{eq:exp_x2}}{=} \bar{r}_0 - A \bar{V}_j \bar{y}_j + A \Theta_j^{\bar{p}} \bar{y}_j - A \Theta_j^{\bar{x}} \, \bar{\boldsymbol{1}}  \notag \\
								&\stackrel{\eqref{eq:AV_BAR}}{=} \bar{r}_0 - \bar{V}_{j+1} \bar{H}_{j+1,j} \bar{y}_j - (\bar{\boldsymbol{V}}_{j+1} - \bar{V}_{j+1}) \bar{\Delta}_{j+1,j} \bar{y}_j+ A \Theta_j^{\bar{p}} \bar{y}_j - A \Theta_j^{\bar{x}} \, \bar{\boldsymbol{1}}  \notag \\
								&\stackrel{\eqref{eq:extra_res2}}{=} \bar{r}_j - (\bar{\boldsymbol{V}}_{j+1} - \bar{V}_{j+1}) \bar{\Delta}_{j+1,j} \bar{y}_j + A \Theta_j^{\bar{p}} \bar{y}_j - A \Theta_j^{\bar{x}} \, \bar{\boldsymbol{1}}.
\end{align}
Expression \eqref{eq:extra_res3} indicates that the gap between $\bar{\bold{r}}_j$ and $\bar{r}_j$ critically depends on the basis vector gaps $\bar{\boldsymbol{V}}_{j+1} - \bar{V}_{j+1}$, which are governed by the inverse of the basis transformation matrix $\bar{G}_j$, see Section \ref{sec:local}. Linking back to expression \eqref{eq:resgap}, it is indeed clear that the second term in the expression for the gap $\bar{\bold{r}}_j - \bar{r}_j $ can also be described by the basis vector gaps, since \eqref{eq:extra_res3} implies that
\begin{equation}
	\|b-A\bar{x}_j\| - |\bar{\zeta}_j| 
		\leq \| (\bar{\boldsymbol{V}}_{j+1} - \bar{V}_{j+1}) \bar{\Delta}_{j+1,j} \bar{y}_j + A \Theta_j^{\bar{p}} \bar{y}_j - A \Theta_j^{\bar{x}} \, \bar{\boldsymbol{1}} \|.
\end{equation}

\subsection{Local rounding error behavior in preconditioned p($\ell$)-CG} \label{sec:preconditioned}
 
In the context of the solution of large-scale systems of equations preconditioning plays a key role in the development of efficient iterative methods. The inclusion of a preconditioner in the p($\ell$)-CG algorithm is (theoretically) straightforward. The preconditioned version of the p($\ell$)-CG algorithm is shown in Algorithm \ref{algo:pplCG}. Note that in contrast to Algorithm \ref{algo:PIPELCG} the basis vectors $V_j$ and $Z_j$ denote the preconditioned versions of the basis vectors in Algorithm \ref{algo:pplCG}. Furthermore, a recurrence relation is introduced for the unpreconditioned auxiliary basis vector $\hat{z}_{j+1} = M z_{j+1}$, which is required to compute the next preconditioned auxiliary basis vector $z_{j+1}$ through the \textsc{spmv}.

The most important observation to be made from Algorithm \ref{algo:pplCG} in the context of rounding error propagation is the following: the recurrences for the preconditioned variables (i.e.:~$v_j$ and $z_j$) do not explicitly use the unpreconditioned variables (i.e.:~$\hat{z}_j$) and vice versa. This implies that the behavior of rounding errors in the basis $V_j$ throughout Algorithm \ref{algo:pplCG} is described by expression \eqref{eq:gap_plcg} in analogy to the unpreconditioned case. 
One notable difference lies in the computation of the entries of the matrix $G_j$. Given the basis vectors $V_{i-l+1}$, $\hat{Z}_{i+1}$ and $Z_{i+1}$, the Euclidean dot product used in Algorithm \ref{algo:PIPELCG} is replaced by the $M$-dot product in Algorithm \ref{algo:pplCG}. For $i < l$ we have
\begin{equation} \label{eq:gdotpr1prec}
g_{j,i+1} = (z_{i+1},z_{j})_M = (\hat{z}_{i+1},z_{j}) , \qquad j=0,\ldots,i+1,
\end{equation}
and for $i \geq l$ it holds that
\begin{equation} \label{eq:gdotpr2prec}
g_{j,i+1}=\left\{ 
  \begin{matrix}
	  (z_{i+1},v_{j})_M = (\hat{z}_{i+1},v_{j}), & j=\max(0,i-2l+1),\ldots,i-l+1, \\ 
	  (z_{i+1},z_{j})_M = (\hat{z}_{i+1},v_{j}), & j=i-l+2,\ldots,i+1.
	\end{matrix}
\right. 
\end{equation}
We remark for completeness that the bounds on the local rounding errors $\xi_j^{\bar{v}}$ and $\xi_j^{\bar{z}}$ also include the norm of the preconditioner where required. 

\begin{algorithm}[t]
{\footnotesize
\caption{Preconditioned p($l$)-CG \hfill \textbf{Input:} $A$, $M^{-1}$, $b$, $x_0$, $l$, $m$, $\sigma_0,\ldots,\sigma_{l-1}$, $\tau$}\label{algo:pplCG}
\begin{algorithmic}[1]
\State $\hat{r}_{0} = b-Ax_{0}; ~ r_{0}:=M^{-1}\hat{r}_{0} ;$ $v_{0}:= r_{0}/{\|r_{0}\|}_M; \hat{z}_{0}:= \hat{r}_{0}/{\|r_{0}\|}_M;$ $z_{0}:=v_{0}; ~ g_{0,0}:=1;$
\For {$i=0,\ldots, m+l$}
\State $\hat{z}_{i+1}:=\left\{ \begin{matrix}Az_{i}-\sigma_{i}\hat{z}_{i} & i<l \\ Az_{i}, & i \geq l \end{matrix}\right.$
\State $z_{i+1} = M^{-1} \hat{z}_{i+1};$
\State $a:=i-l$
\If {$a\geq 0$}
\State $g_{j,a+1} := (g_{j,a+1}-\sum_{k=a+1-2l}^{j-1}g_{k,j}g_{k,a+1})/g_{j,j}; \qquad j=a-l+2,\ldots,a$  
\State $g_{a+1,a+1}:= \sqrt{g_{a+1,a+1}-\sum_{k=a+1-2l}^{a}g_{k,a+1}^2};$
\State \# Check for breakdown and restart if required
\If {$a<l$}
\State $\gamma_{a}:=(g_{a,a+1}+\sigma_{a}g_{a,a}-\delta_{a-1}g_{a-1,a})/g_{a,a};$
\State $\delta_{a}:=g_{a+1,a+1}/g_{a,a};$
\Else
\State $\gamma_{a}:=(g_{a,a}\gamma_{a-l}+g_{a,a+1}\delta_{a-l}-\delta_{a-1}g_{a-1,a})/g_{a,a};$
\State $\delta_{a}:=(g_{a+1,a+1}\delta_{a-l})/g_{a,a};$
\EndIf
\State \textbf{end if}
\State $v_{a+1} := (z_{a+1} - \sum_{j=a-2l+1}^{a} g_{j,a+1} v_{j})/g_{a+1,a+1};$
\State $\hat{z}_{i+1} := (\hat{z}_{i+1} - \gamma_{a} \hat{z}_{i}- \delta_{a-1} \hat{z}_{i-1})/\delta_{a};$
\State $z_{i+1} := (z_{i+1} - \gamma_{a} z_{i}- \delta_{a-1} z_{i-1})/\delta_{a};$
\EndIf
\State \textbf{end if}
\If {$a<0$}
\State $g_{j,i+1}:=(\hat{z}_{i+1},z_{j}); \qquad j=0,\ldots,i+1$
\Else 
\State $g_{j,i+1}:=\left\{ \begin{matrix}(\hat{z}_{i+1},v_{j}); & j=\max(0,i-2l+1),\ldots,a+1 \\ (\hat{z}_{i+1},z_{j});  &j=a+2,\ldots,i+1 \end{matrix}\right.$
\EndIf
\State \textbf{end if}
\If {$a=0$}
\State $\eta_{0}:=\gamma_{0};$
\State $\zeta_{0}:={\|r_{0}\|}_M;$
\State $p_{0}:=v_0/\eta_0;$
\Else \, \textbf{if} {$a\geq 1$} \textbf{then}
\State $\lambda_{a}:=\delta_{a-1}/\eta_{a-1};$ 
\State $\eta_{a}:=\gamma_{a}-\lambda_{a}\delta_{a-1};$
\State $\zeta_{a}=-\lambda_{a}\zeta_{a-1};$
\State $p_{a}=(v_{a}-\delta_{a-1}p_{a-1})/\eta_{a};$
\State $x_{a}=x_{a-1}+\zeta_{a-1}p_{a-1};$
\If {$|\zeta_a|/\|r_0\|_M < \tau$}
\State RETURN;
\EndIf
\State \textbf{end if} 
\EndIf
\State \textbf{end if}
\EndFor 
\State \textbf{end for}
\end{algorithmic}
}
\end{algorithm}

\section{Further analysis of the local rounding error behavior} \label{sec:further}

The matrices $\bar{G}^{-1}_j$ fulfill a crucial role in the propagation of local rounding errors in p($\ell$)-CG, see \eqref{eq:gap_plcg2}-\eqref{eq:gap_plcg}. In this section we aim to establish useful practical bounds on the maximum norm of $G^{-1}_j$, i.e.
\begin{equation}
	\| G^{-1}_j \|_{\max} = \max_k \max_l | G^{-1}_j(k,l) |,
\end{equation}
in exact arithmetic to eventually obtain practical insights in its finite precision variant $\bar{G}^{-1}_j$.

\subsection{Establishing upper bounds on $\|{G}^{-1}_j\|_{\max}$} \label{sec:bounds}

The inverse of the banded matrix $G_j$ is an upper triangular matrix, which can be expressed as
\begin{equation} \label{eq:expan}
	G^{-1}_j = \sum_{k=0}^{j-1} \left( - \Lambda^{-1} G_j^{\triangle} \right)^{k} \Lambda^{-1},
\end{equation}
where $\Lambda := [\delta_{mk} g_{m,k}]$ contains the diagonal of $G_j$, and $G_j^{\triangle} := G_j - \Lambda$ is the strictly upper triangular matrix obtained by setting the diagonal of $G_j$ to zero.
\begin{lemma} \label{lemma:lemma_0}
	Let the Krylov subspace basis transformation matrix $G_j$ be defined by $Z_j = G_j V_j$ for $1 \leq j \leq i-l+1$. Then it holds that
	\begin{equation}
		\|G_j\|_{\max} \leq  \|P_l(A)\|. \label{eq:est}
	\end{equation}
\end{lemma}
\begin{proof}
	Since for any $0 \leq m \leq i-l$ the vector $v_m$ is normalized, i.e. $\|v_m\| = 1$, the following bound on the entries of $G_j$ holds
	\begin{equation} \label{eq:est2}
	  |g_{m,k}| = |(z_k,v_m)| \leq \|z_k\| \leq \|P_l(A)\| = \lambda_{\max}(P_l(A)), \quad 0 \leq m \leq i-l, \quad 0 \leq k \leq i,
	\end{equation}
where $\lambda_{\max}(P_l(A))$ is the largest eigenvalue of the Hermitian positive definite matrix $P_l(A)$.
\end{proof}

We now refine the above bounds on $\|G^{-1}_j\|_{\max}$ 
using the relation $Z_j = V_j G_j$ and the Lanczos relation $A V_j = V_j H_{j,j}$ for $1 \leq j \leq i-l$, where $H_{j,j}$ is the symmetric tridiagonal Lanczos matrix.

\begin{lemma} \label{lemma:lemma_1}
Let $j \geq l+1$ and let $V_{l+1:j} = [v_l,\ldots,v_{j-1}]$  and $Z_{l+1:j} = [z_l,\ldots,z_{j-1}]$ denote subsets of the Krylov bases $V_j$ and $Z_j$. Let $G_{l+1:j}$ be the principal submatrix of $G_{j}$ that is obtained by removing the first $l$ rows and columns of $G_{j}$. Then
\begin{equation} \label{eq:G_sub}
  G_{l+1:j} = V^T_{l+1:j} \, V_{1:j-l} \, P_l(H_{j-l,j-l}) = \left( P_l(H_{j-l,j-l}) \, V^T_{1:j-l} \, V_{l+1:j} \right)^T,
\end{equation}
where 
\begin{equation}
	V^T_{l+1:j} \, V_{1:j-l} = 
	\left(\begin{array}{ccccc} 
		0&&1&& \\
		&\ddots&&\ddots& \\
		&&\ddots&&1 \\
		&&&\ddots& \\
		&&&&0 \\
	\end{array}\right) \leftarrow (l+1)^{\text{th}}~\text{row}.
\end{equation}
\end{lemma}

\noindent Lemma \ref{lemma:lemma_1} effectively states that for any $j \geq l+1$ the principal submatrix $G_{l+1:j}$ of the basis transformation matrix $G_{j}$ can be obtained by shifting all entries of the matrix $P_l(H_{j,j})$ upward by $l$ places and subsequently selecting the leading $(j-l) \times (j-l)$ block.

\begin{proof}
The proof follows directly from the definition \eqref{eq:z_ex} and the Lanczos relation, which implies $P_l(A)V_j = V_j P_l(H_{j,j})$, and hence
\begin{equation}
  G_{l+1:j} = V^T_{l+1:j} \, Z_{l+1:j} = V^T_{l+1:j} \, P_l(A) \, V_{1:j-l} = V^T_{l+1:j} \, V_{1:j-l} \, P_l(H_{j-l,j-l}).
\end{equation}
\end{proof}

\begin{lemma} \label{lemma:lemma_bound}
Let $j \geq l+1$ and let $G_{l+1:j}$ be the principal submatrix of $G_{j}$ that is obtained by removing the first $l$ rows and columns of $G_{j}$. Let $V_{l+1:j} = [v_l,\ldots,v_{j-1}]$  and $Z_{l+1:j} = [z_l,\ldots,z_{j-1}]$ denote subsets of the bases $V_j$ and $Z_j$ respectively. Then the following bound holds:
\begin{equation} \label{eq:ineq1}
	\| G^{-1}_{l+1:j} \|_{\max} \leq \left( \min_{1 \leq k \leq j-l} \left| P_l(\theta^{j-l}_{k}) \right| \right)^{-1},
\end{equation}
where the scalar $\theta^{j-l}_{k}$ denotes the $k$-th Ritz value, i.e.~the $k$-th eigenvalue of the matrix $H_{j-l,j-l}$, with $1 \leq k \leq j-l$. Moreover, when the classic monomial Newton basis is considered, i.e.~the shifts $\sigma_j$ for $j=0,\ldots, l-1$ are set to zero, the following bound can alternatively be derived:
\begin{equation} \label{eq:ineq2}
 \| G^{-1}_{l+1:j} \|_{\max} \leq \left(|\lambda_{\max}(H_{j-l,j-l}^{-1})|\right)^{l}.
\end{equation}
\end{lemma}

\begin{proof}
The inequality \eqref{eq:ineq1} can be proven as follows:
\begin{align*} 
	\| G^{-1}_{l+1:j} \|_{\max} &\stackrel{\eqref{eq:G_sub}}{=} \|\left(P_l(H_{j-l,j-l}) \, V^T_{1:j-l} \, V_{l+1:j} \right)^{-1} \|_{\max} 
														\leq \|\left(P_l(H_{j-l,j-l}) \, V^T_{1:j-l} \, V_{l+1:j} \right)^{-1} \| \\
														&= \left( \sigma_{\min}\left( P_l(H_{j-l,j-l}) \, V^T_{1:j-l} \, V_{l+1:j} \right) \right)^{-1} \\
														&= \left( \sqrt{ \lambda_{\min}\left( P_l(H_{j-l,j-l}) \, V^T_{1:j-l} \, V_{l+1:j} \, V^T_{l+1:j} \, V_{1:j-l} \, P_l(H_{j-l,j-l}) \right) } \right)^{-1} \\
														&= \left( \sqrt{ \lambda_{\min}\left( (P_l(H_{j-l,j-l})^2)_{l+1:j} \right) } \right)^{-1} 
														\leq \left( \sqrt{ \lambda_{\min}\left( P_l(H_{j-l,j-l})^2 \right) } \right)^{-1} \\
														&= \left( \left| \lambda_{\min}( \, P_l(H_{j-l,j-l}) \, ) \right| \right)^{-1} = \| \left(P_l(H_{j-l,j-l})\right)^{-1} \| = \left( \min_{1 \leq k \leq j-l} \left| P_l(\theta^{j-l}_{k}) \right| \right)^{-1}.
\end{align*}
Here $\sigma_{\min}(\cdot)$ and $\lambda_{\min}(\cdot)$ denote respectively the minimal non-zero singular value and eigenvalue of the given matrix, and $(P_l(H_{j-l,j-l})^2)_{l+1:j}$ is the principal submatrix of the squared matrix polynomial $P_l(H_{j-l,j-l})^2$ obtained by removing the latter's first $l$ rows and columns. The last inequality in the derivation above follows from the Courant-Fischer minimax principle, see e.g.~\cite{horn1990matrix}. 

The inequality \eqref{eq:ineq2} follows directly by setting the shifts $\sigma_j$ $(j=0,\ldots, l-1)$ to zero:
\begin{equation*}
  \| G^{-1}_{l+1:j} \|_{\max} \leq \| (H_{j-l,j-l}^{l})^{-1} \|_{2} = \| (H_{j-l,j-l}^{-1})^{l} \|_{2} = \left(|\lambda_{\max}(H_{j-l,j-l}^{-1})|\right)^{l},
\end{equation*}
which completes the proof.
\end{proof}

The derivation of the upper bounds \eqref{eq:ineq1}-\eqref{eq:ineq2} in Lemma \ref{lemma:lemma_bound} has been restricted to the principal submatrix $G_{l+1:j}$ to avoid notational difficulties due to the polynomial definition in the first $l$ iterations, see expression \eqref{eq:z_ex}, but can be generalized to the entire matrix $G_{j}$.

\subsection{Interpreting the bounds
} \label{sec:interpreting}

We now provide the reader with additional insights on how to interpret the bounds derived in Section \ref{sec:bounds} in practice. We first sketch a heuristic argumentation based on Lemma \ref{lemma:lemma_0}, inequality \eqref{eq:est2} to indicate that the maximum norm of $\bar{G}^{-1}_j$ may be much larger than one, possibly leading to the amplification of local rounding errors in p($\ell$)-CG.

Note that the matrix $G_j$ is not necessarily diagonally dominant. Indeed, since the norms $\|z_k\|$ are not guaranteed to be monotonically decreasing in p($\ell$)-CG, it may occur that $|g_{m,k}| \geq |g_{m,m}|$ for some $0 \leq m < k \leq i$. Consequently, if the bound in \eqref{eq:est} is tight, expression \eqref{eq:expan} suggests that when $\|P_l(A)\|$ is large, $\|G^{-1}_j\|_{\max}$ may be (much) larger than one.  Assuming that this heuristic argumentation may be extended to the finite precision framework, $\|\bar{G}^{-1}_{j}\|_{\max}$ may be large -- depending on the spectral properties of $P_l(A)$ -- and the corresponding local rounding errors in the expressions \eqref{eq:gap_plcg2}-\eqref{eq:gap_plcg} are possibly amplified throughout the algorithm. This is illustrated by the numerical experiments in Section \ref{sec:numerical}, see Fig.~\ref{fig:figure1}-\ref{fig:figure3}. 
	
\begin{remark} \label{remark:shifts}
	\underline{Influence of the Krylov basis choice (Lemma \ref{lemma:lemma_0}).} If the stabilizing shifts $\sigma_i$ ($0 \leq i < l$) are chosen as roots of the degree $l$ Chebyshev polynomial in the Newton basis as introduced for the p($\ell$)-CG algorithm in \cite{ghysels2013hiding,cornelis2017communication}, they effectively minimize the 2-norm of $P_l(A)$, see \cite{greenbaum1994gmres,faber2010chebyshev}. Expressions \eqref{eq:expan}-\eqref{eq:est} suggest that this choice of shifts may indeed reduce the effect of local rounding error propagation. This is illustrated by the numerical experiments in Section \ref{sec:numerical}, Fig.~\ref{fig:figure1} and \ref{fig:figure3}. However, even with the `optimal' choice for the shifts, the experiments show that the amplification of local rounding errors can not be precluded entirely, in particular when the pipeline depth $l$ is large. 
\end{remark}

	Note that the bound $\|\bar{G}_j\|_{\max} \leq \|P_l(A)\|$ from \eqref{eq:est} may not hold in finite precision arithmetic due to loss of basis orthogonality, which also adds to the amplification of local rounding errors. We refer to the literature \cite{greenbaum1989behavior,greenbaum1992predicting,strakovs2002error} for more information on this topic. Due to this effect, the finite precision propagation matrix $\|\bar{G}^{-1}_j\|_{\max}$ may even be significantly larger than $\|G^{-1}_j\|_{\max}$ in practice. A profound study of the loss of orthogonality in pipelined Krylov subspace methods is however beyond the scope of this work.

The following intuitive insights on the effects of the pipeline depth and the iteration on rounding error behavior can be derived from Lemma \ref{lemma:lemma_bound} and, notably, the bounds 
\eqref{eq:ineq1}-\eqref{eq:ineq2}.


\begin{remark} \label{remark:depth}
\underline{Influence of the pipeline depth (Lemma \ref{lemma:lemma_bound}).} Let the iteration index $j > l$ be fixed and let us without loss of generality consider the monomial basis for the Krylov subspace 
in this argumentation. Suppose that for p($\ell$)-CG with pipeline length $l = 1$ it holds that $\|G^{-1}_{l+1:j}\|_{\max} > 1$, indicating that local rounding error amplification occurs in (or before) iteration $j$. From expression \eqref{eq:ineq2} it then follows that $|\lambda_{\max}(H_{j-l,j-l}^{-1})| > 1$. Consequently, for p($\ell$)-CG methods with pipeline lengths $l > 1$ it holds that
  $(|\lambda_{\max}(H_{j-l,j-l}^{-1})|)^l > |\lambda_{\max}(H_{j-l,j-l}^{-1})| > 1$,
signaling that the propagation of local rounding errors up to iteration $j$ may 
be even more dramatic for larger values of $l$. The results in Table \ref{tab:bounds} substantiate this premise.
\end{remark}

\begin{remark} \label{remark:index}
\underline{Influence of increasing iteration index (Lemma \ref{lemma:lemma_bound}).} Let us assume that the pipeline length $l$ is fixed and the monomial basis for the Krylov subspace is used. As the iteration index $j$ increases in the p($\ell$)-CG algorithm, the number of Ritz values $\theta^{j}_{k}$ $(1 \leq k \leq j)$ amounts, gradually approximating the spectrum of the matrix $A$. For $j_1 \leq j_2$ this implies $\min_{1\leq k \leq j_1}|\theta^{j_1}_{k}| \geq \min_{1\leq k \leq j_2}|\theta^{j_2}_{k}|$, such that
\begin{equation}
	\left(\min_{1 \leq k \leq j_1} \left(\left| \theta^{j_1}_{k} \right|^{l}\right)\right)^{-1} \leq \left(\min_{1 \leq k \leq j_2} \left( \left| \theta^{j_2}_{k} \right|^{l} \right) \right)^{-1}.
\end{equation}
This result suggests that the bound \eqref{eq:ineq1} on the norm $\| G^{-1}_{l+1:j} \|_{\max}$ is monotonically increasing with $j$, and that the impact of local rounding errors can thus be expected to become more pronounced as the iteration proceeds. This intuitive observation is supported by the numerical experiments reported in Table \ref{tab:bounds} and Fig.~\ref{fig:figure3}.
\end{remark}

\section{Countermeasures to local rounding error amplification in p($\ell$)-CG} \label{sec:countermeasures}

The analysis in Sections \ref{sec:local} and \ref{sec:further} shows that the multi-term recurrence relation \eqref{eq:vbar_rec} for the basis vector $\bar{v}_{j+1}$ forms the main cause for the amplification of local rounding errors throughout the p($\ell$)-CG algorithm. 
A straightforward countermeasure to this rounding error behavior thus consists of replacing the recurrence relation \eqref{eq:vbar_rec} by the original Lanczos recurrence for $\bar{v}_{j+1}$, i.e.
\begin{equation} \label{eq:vbar_stab}
	\bar{v}_{j+1} = ( A\bar{v}_{j} - \bar{\gamma}_j \bar{v}_j - \bar{\delta}_{j-1} \bar{v}_{j-1} ) / \bar{\delta}_j + \psi^{\bar{v}}_{j+1}, \qquad 0 \leq j < i - l,
\end{equation} 
where the local rounding error $\psi^{\bar{v}}_{j+1}$ is bounded by
\begin{equation*}
	\| \psi^{\bar{v}}_{j+1} \| \leq \left( (2+\mu\sqrt{n}) \frac{\|A\|}{|\bar{\delta}_j|} \|\bar{v}_j\| + 3 \frac{|\bar{\gamma}_j|}{|\bar{\delta}_j|} \|\bar{v}_j\| + 3 \frac{|\bar{\delta}_{j-1}|}{|\bar{\delta}_j|} \|\bar{v}_{j-1}\| \right) \epsilon.
\end{equation*} 
Using expression \eqref{eq:vbar_stab} instead of relation \eqref{eq:vbar_rec} in Algorithm \ref{algo:PIPELCG} significantly decreases the number of terms in the recurrence relation. Furthermore, it removes the explicit dependency of $\bar{v}_{j+1}$ on the auxiliary basis vector $\bar{z}_{j+1}$.
However, the recurrence relation \eqref{eq:vbar_stab} requires to compute an additional \textsc{spmv} $A\bar{v}_{j}$ in the p($\ell$)-CG algorithm, leading to an increase in computational cost. This stabilization technique bears similarity to the concept of \emph{residual replacement}, i.e.~the replacement of the recursive residual $\bar{r}_j$ by the explicit computation of $b-A\bar{x}_j$, which was suggested by several authors as a countermeasure to local rounding error propagation in various multi-term recurrence variants of CG \cite{van2000residual,carson2014residual,cools2018analyzing}. 

The alternative recurrence relation \eqref{eq:vbar_stab} for the basis vectors of $\bar{V}_j$ can be written in matrix form as
\begin{equation}
	A\bar{V}_j = \bar{V}_{j+1} \bar{H}_{j+1,j} - \Psi^{\bar{v}}_{j+1} \bar{\Delta}_{j+1,j} , \qquad 1 \leq j \leq i-l,
\end{equation}
where the local rounding errors are $\Psi^{\bar{v}}_j = [\psi^{\bar{v}}_0,\ldots,\psi^{\bar{v}}_{j-1}]$. The gap between the true basis vector $\bar{\bold{v}}_{j+1}$ and the computed basis vector $\bar{v}_{j+1}$ then reduces to
\begin{equation} \label{eq:gap_stab}
	\mathcal{F}_{j+1} = \bar{\boldsymbol{V}}_{j+1} - \bar{V}_{j+1} = - \Psi^{\bar{v}}_{j+1} 
\end{equation}
Hence, using the recurrence \eqref{eq:vbar_stab} for $\bar{v}_{j+1}$, local rounding errors are 
not amplified, and the impact of local rounding errors is comparable to the original CG algorithm, see Fig.~\ref{fig:figure5}.

\begin{remark}
  The use of the recurrence relation \eqref{eq:vbar_stab} requires the computation of one additional \textsc{spmv}. Although this \textsc{spmv} may also be overlapped by the pipelined global reduction phases, the use of this recurrence relation induces a significant extra computational cost. 
	In \cite{carson2014residual} and \cite{cools2018analyzing} similar stabilizing techniques were presented for the s-step CG method and one-step pipelined CG method respectively by performing a residual replacement in only a selected number of iterations.
  However, experiments have shown that replacing the recurrence relation for $\bar{v}_{j+1}$ by \eqref{eq:vbar_stab} in a limited number of p($\ell$)-CG iterations rapidly reintroduces amplified local rounding errors, and hence is not a robust countermeasure for stabilizing the p($\ell$)-CG algorithm. When applications demand a high precision $\bar{v}_{j+1}$ can thus be computed using \eqref{eq:vbar_stab} in every iteration, however; implementing this countermeasure implies a trade-off between numerical attainable accuracy and computational cost, see Fig.~\ref{fig:figure6}.
\end{remark}

\section{Numerical experiments} \label{sec:numerical}

We consider the $200 \times 200$ uniform second order central finite difference discretization of a 2D Poisson problem on the unit domain $[0,1]^2$ for illustration purposes in this section. Note that this relatively simple problem is significantly ill-conditioned and serves as an adequate tool for demonstrating the error analysis in this work. The right-hand side of the system is $b = A\hat{x}$ with $\hat{x} = 1/\sqrt{n}$, unless explicitly stated otherwise, and the initial guess is $x_0 = 0$.

\begin{figure}
\begin{center}
\includegraphics[width=0.47\textwidth]{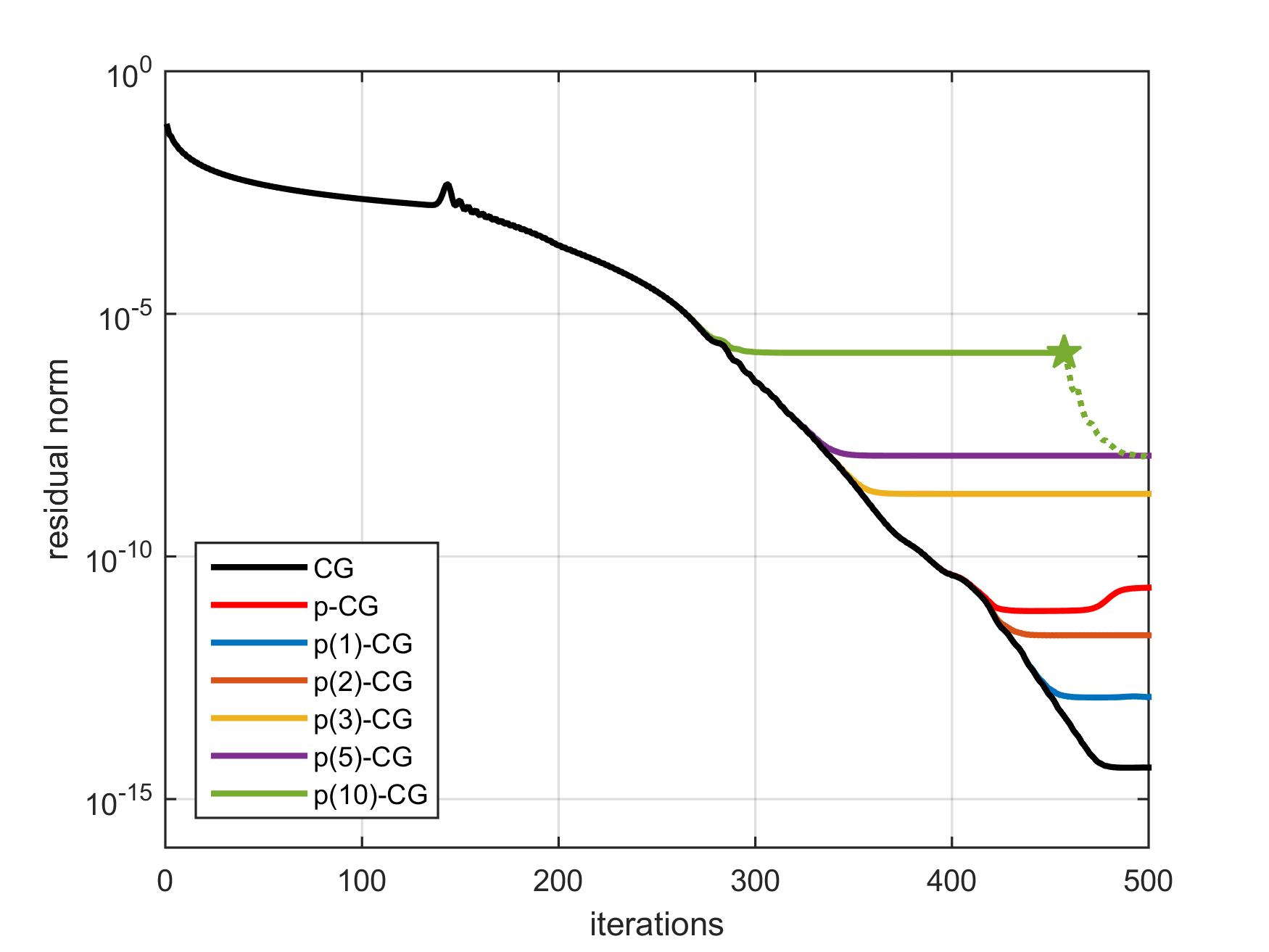} 
\includegraphics[width=0.47\textwidth]{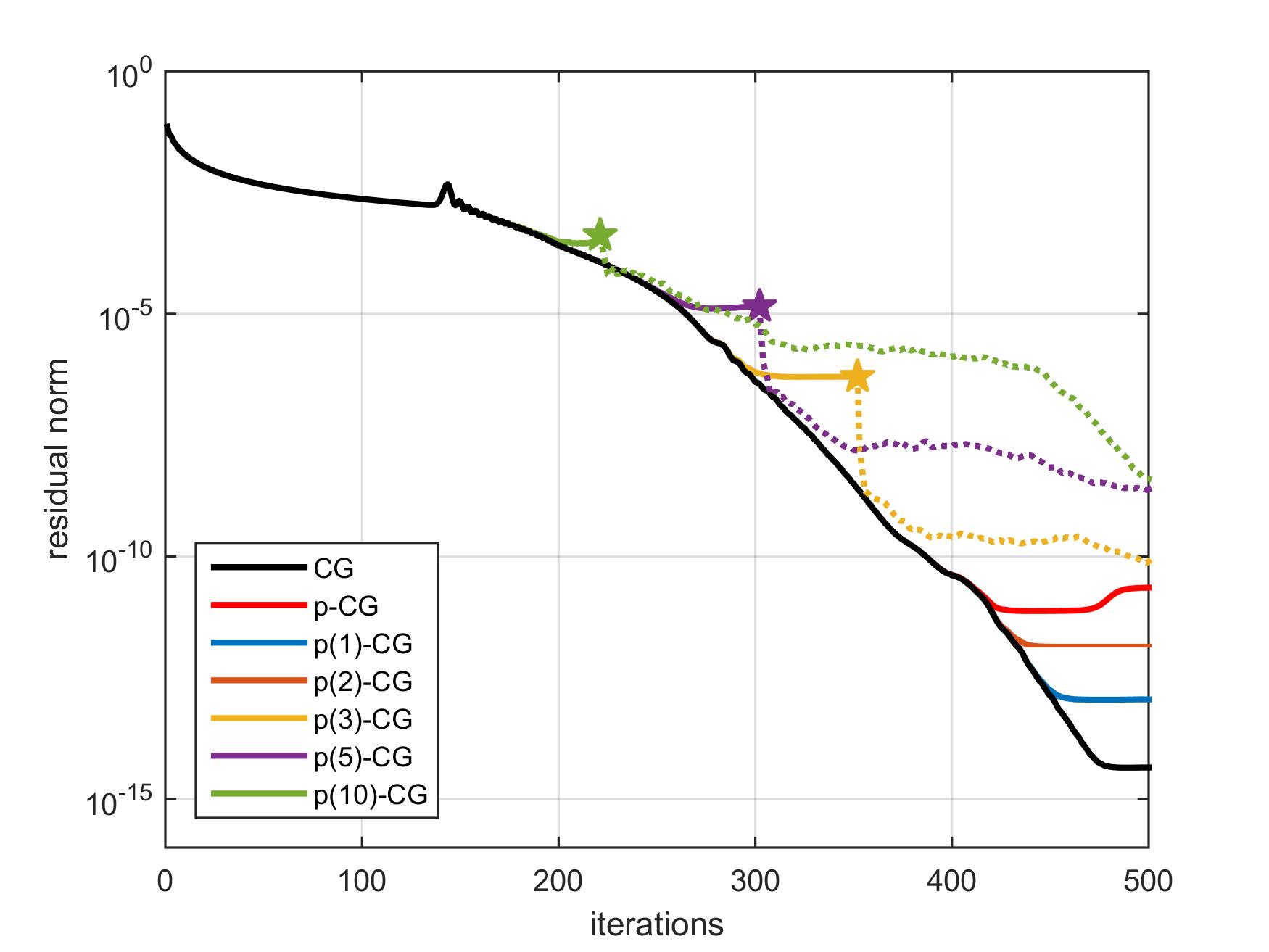}
\end{center}
\caption{Comparison of true residual norm history $\|b-A\bar{x}_j\|$ for different CG variants for a 2D Poisson problem with 200 $\times$ 200 unknowns. The stabilizing shifts $\sigma_i$ for p($\ell$)-CG are based on the degree $l$ Chebyshev polynomial on the interval [0,8] (left) (optimal shift choices) and the perturbed interval [0,8*1.005] (right) (slightly sub-optimal shifts). Square root breakdowns in p($\ell$)-CG are indicated by a $\bigstar$ symbol (followed by explicit iteration restart).}
\label{fig:figure1}
\end{figure}

Figure \ref{fig:figure1} shows the norm of the true residual $b-A\bar{x}_j$ as a function of iterations for different variants of the CG algorithm, including classic CG (Alg.~\ref{algo:CG}), pipelined CG (Alg.~\ref{algo:PIPECG}) and p($\ell$)-CG (Alg.~\ref{algo:PIPELCG}).\footnote{The explicit computation of the true residual norm $\|b-A\bar{x}_j\|$ was added to the p($\ell$)-CG algorithm for illustration purposes in the numerical experiments reported in Fig.~\ref{fig:figure1}, Fig.~\ref{fig:figure2}, Fig.~\ref{fig:figure4}, Fig.~\ref{fig:figure5} and Fig.~\ref{fig:figure7}.} The p($\ell$)-CG algorithms use Chebyshev shifts $\sigma_i$ that are defined as
\begin{equation*}
  \sigma_{i} = \frac{\lambda_{\max}+\lambda_{\min}}{2}+\frac{\lambda_{\max}-\lambda_{\min}}{2} \cos\left(\frac{(2i+1)\pi}{2l}\right), \qquad i=0,\ldots,l-1,
\end{equation*}
where all eigenvalues of $A$ are assumed to be located in the interval $[\lambda_{\min},\lambda_{\max}]$. 
We refer to the original paper on pipelined GMRES \cite{ghysels2013hiding} and the work by Hoemmen \cite{hoemmen2010communication} for additional information.
The influence of the pipeline length $l$ and the basis choice are illustrated in Fig.~\ref{fig:figure1}.

\begin{figure}
\begin{center}
\begin{tabular}{cc}
\includegraphics[width=0.47\textwidth]{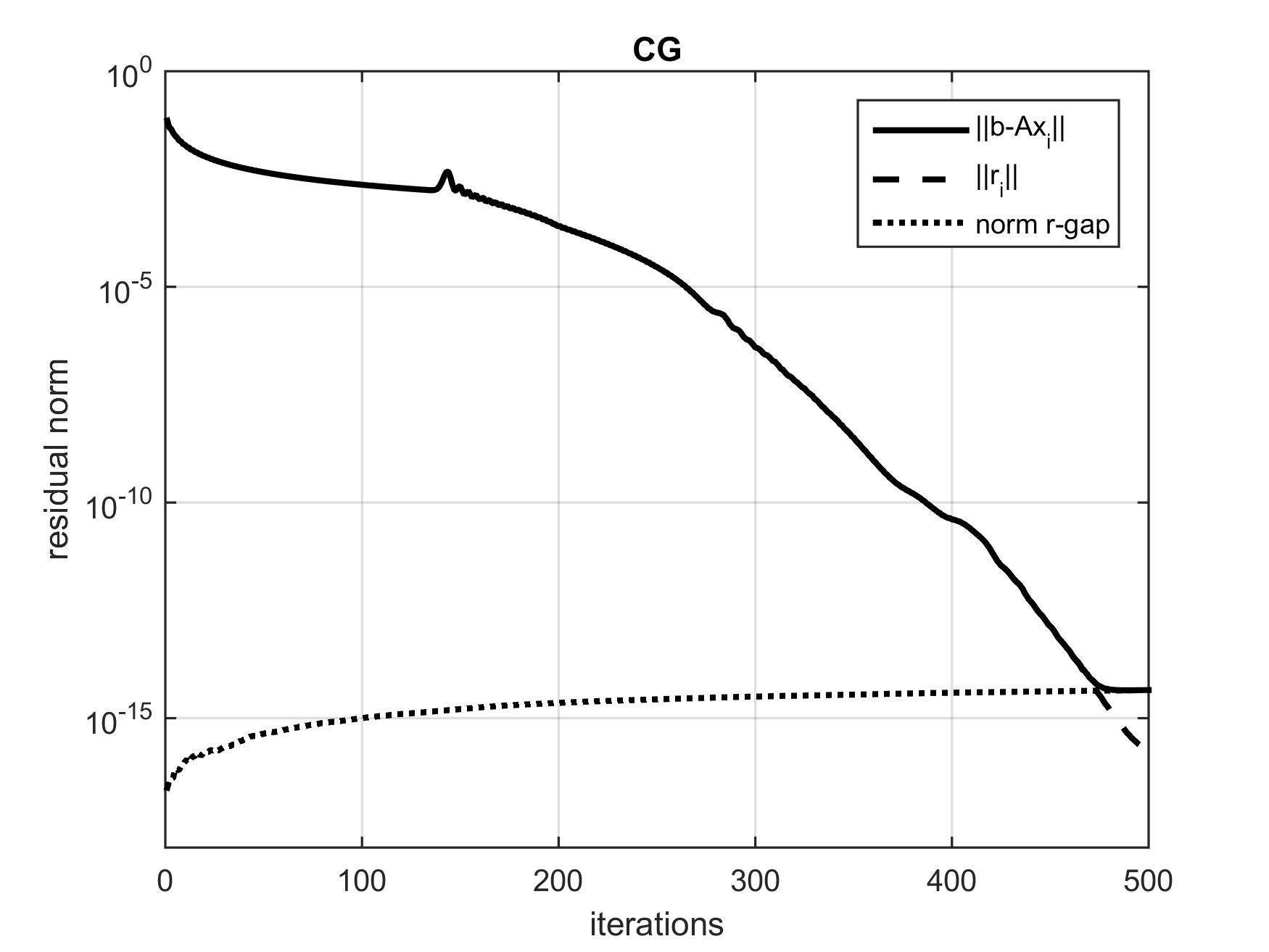} &
\includegraphics[width=0.47\textwidth]{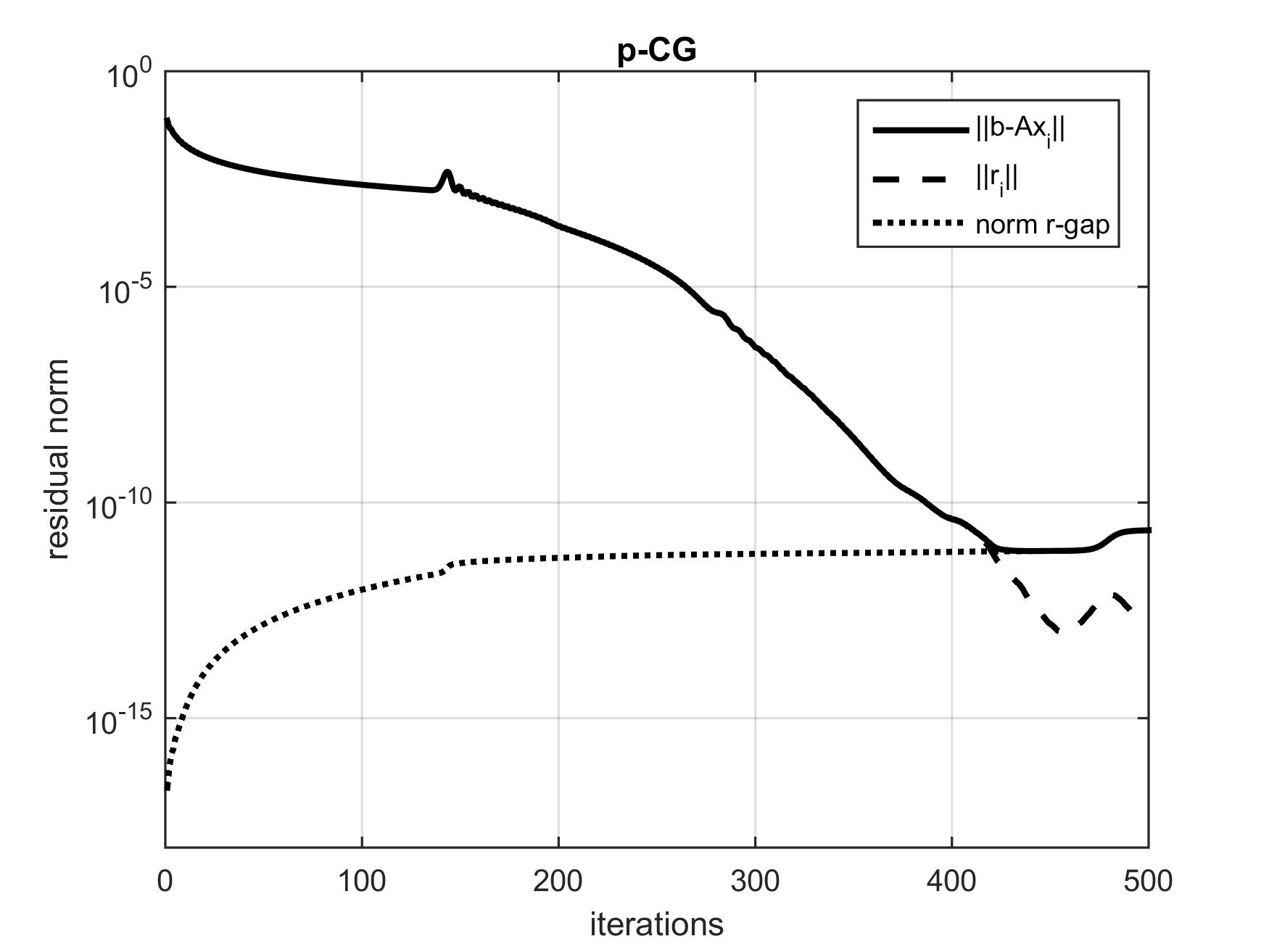} \\
\includegraphics[width=0.47\textwidth]{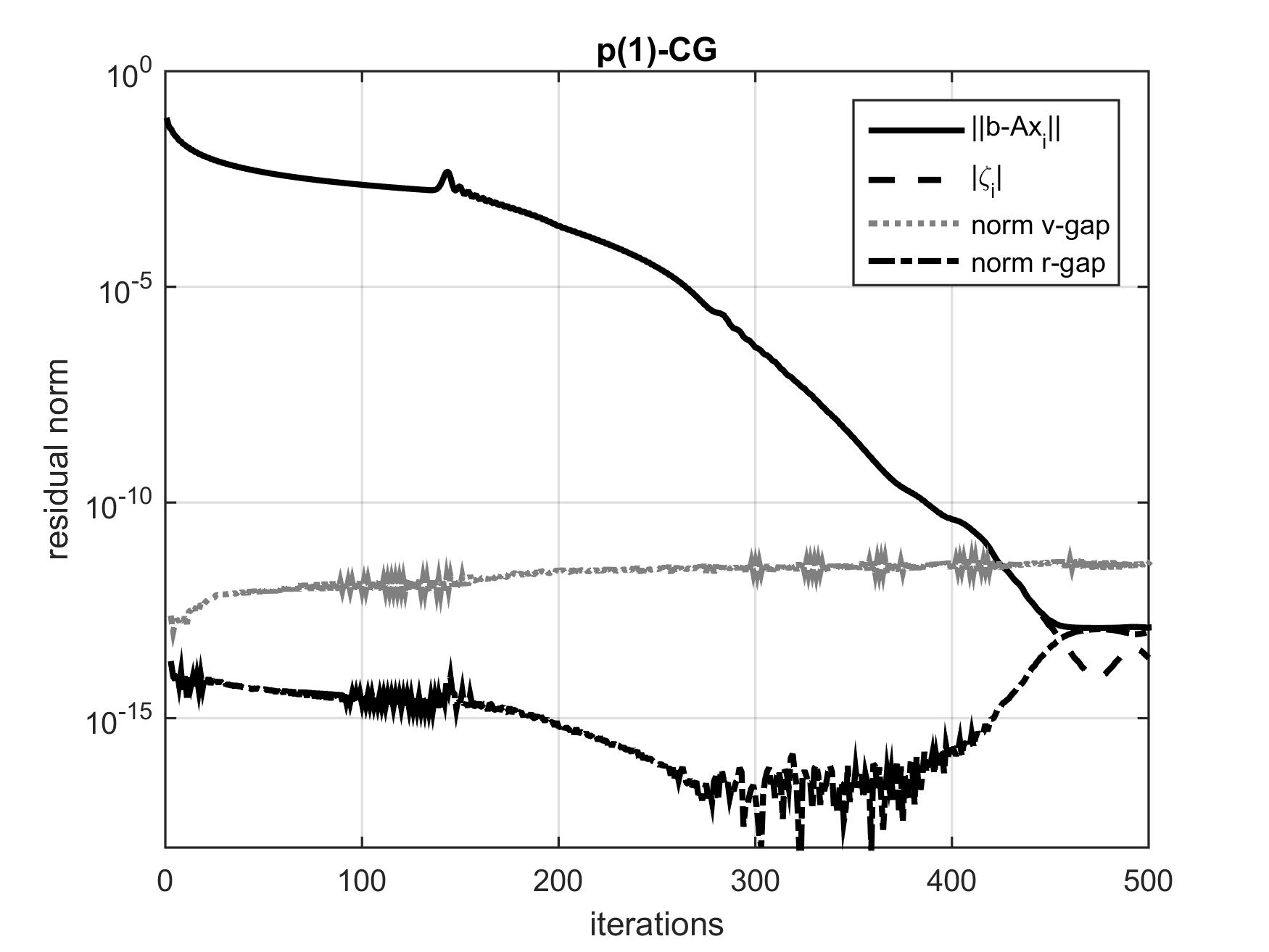} &
\includegraphics[width=0.47\textwidth]{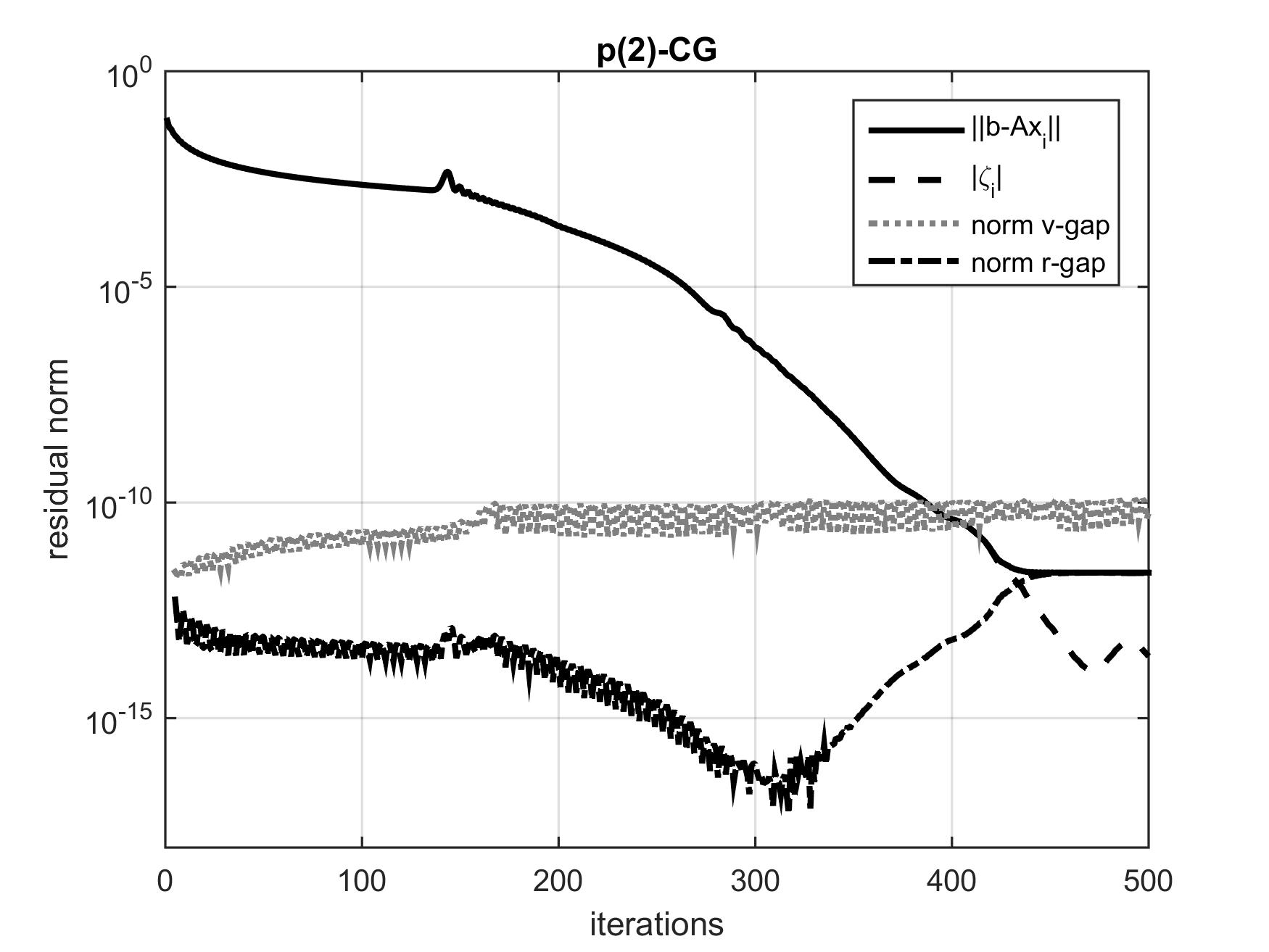} \\
\includegraphics[width=0.47\textwidth]{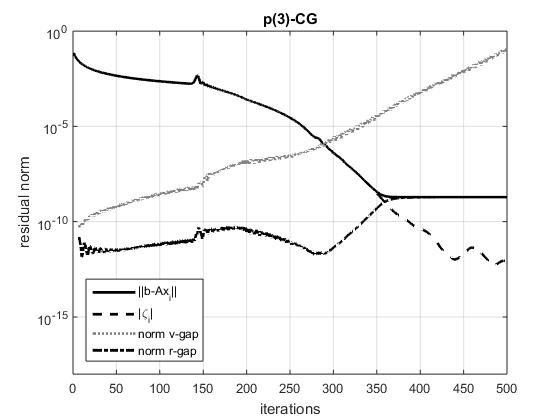} &
\includegraphics[width=0.47\textwidth]{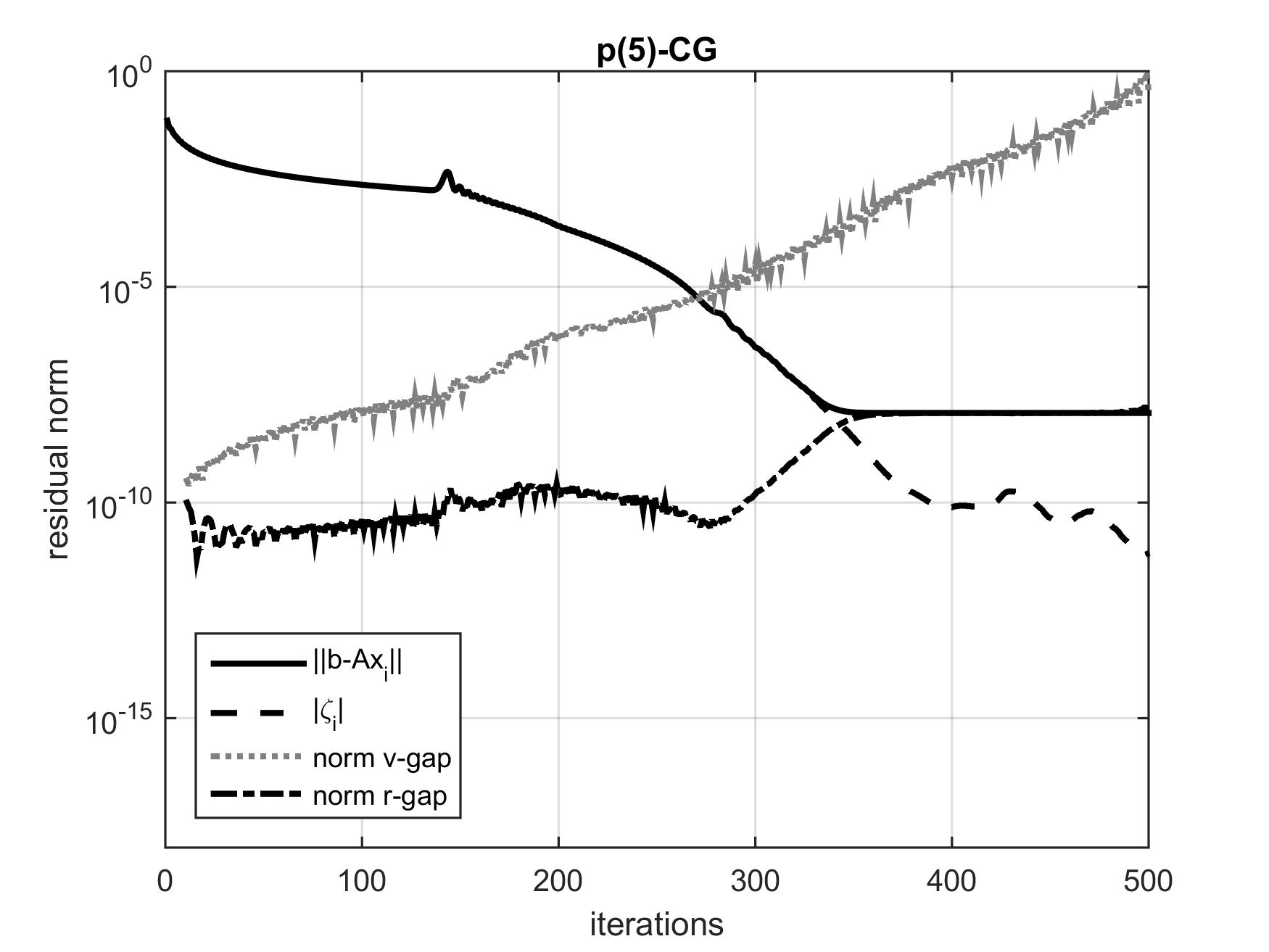} 
\end{tabular}
\end{center}
\caption{True residual norm $\|b-A\bar{x}_{j+1}\|$, computed residual norm $\|\bar{r}_{j+1}\|$, gap norm $\|f_{j+1}\| = \|\bar{\bold{v}}_{j+1} - \bar{v}_{j+1}\|$ and gap norm $\|\bar{\bold{r}}_{j+1}-\bar{r}_{j+1}\|$ for different CG variants on a 2D Poisson problem with 200 $\times$ 200 unknowns corresponding to Fig.~\ref{fig:figure1} (left). For CG and p-CG (top) the norm of the residual gap $(b-A\bar{x}_{j+1}) - \bar{r}_{j+1}$ is displayed, where $\bar{r}_{j+1}$ is computed using \eqref{eq:recs_xandr}. For p($\ell$)-CG with $l = 1,2,3,5$ (middle and bottom) the norms of the gaps $\bar{\bold{v}}_{j-l+1} - \bar{v}_{j-l+1}$ and $\bar{\bold{r}}_{j-l+1} - \bar{r}_{j-l+1}$ are shown, where $\bar{\bold{v}}_{j-l+1}$ satisfies \eqref{eq:v_arnoldi}, $\bar{v}_{j-l+1}$ is computed using the recurrence \eqref{eq:vbar_rec} and $\bar{\bold{r}}_{j-l+1} - \bar{r}_{j-l+1}$ is defined by \eqref{eq:resgap} using $\bar{\bold{v}}_{j-l+1}$ and $\bar{v}_{j-l+1}$.}
\label{fig:figure2}
\end{figure}

Figure \ref{fig:figure2} shows the norm of the true residual $b-A\bar{x}_j$ and the norm of the computed residual $\bar{r}_j$ (computed as $|\zeta_j|$ in p($\ell$)-CG, see \eqref{eq:resnorm}) for various CG variants. It also displays the norm of the gap $f_j = (b-A\bar{x}_j) - \bar{r}_j$ for CG/p-CG, and the norm of the gap $f_j = \bar{\bold{v}}_j - \bar{v}_j$ and the norm of the residual gap $\bar{\bold{r}}_{j}-\bar{r}_{j}$ for p($\ell$)-CG. For p($\ell$)-CG the gap between $\bar{\bold{v}}_j$ and $\bar{v}_j$ increases dramatically as the iteration proceeds, particularly for large values of $l$, leading to significantly reduced maximal attainable accuracy as indicated by the related residual gap norms. The following true residuals norms $\|b-A\bar{x}_j\|$ are attained after 500 iterations: $4.47\text{e-}15$ (CG), $2.28\text{e-}11$ (p-CG), 1.27\text{e-}13 (p($1$)-CG), $2.37\text{e-}12$ (p($2$)-CG), $1.94\text{e-}09$ (p($3$)-CG), $1.19\text{e-}08$ (p($5$)-CG).

\begin{table}[t]
\centering
\footnotesize
\begin{tabular}{| c || r | r | r | r | r | r | }
\hline 
	$\|\bar{G}^{-1}_{j}\|_{\max}$ &  &  &  &  &  &  \\
	$\|(P_l(\bar{H}_{j,j}))^{-1}\|$ & $l = 1$ & $l = 2$ & $l = 3$ & $l = 4$ &$l = 5$ & $l = 10$ \\
\hline \hline
	$j=10$  & 2.7e+00 & 5.8e+01 & 1.6e+02 & 3.3e+02 & 5.3e+02 & 1.0e+03 \\
	        & 8.2e+00 & 1.2e+02 & 3.2e+02 & 6.8e+02 & 1.0e+03 & 1.6e+03 \\ \hline
	$j=50$  & 2.8e+00 & 5.8e+01 & 1.6e+02 & 3.3e+02 & 5.7e+02 & 3.2e+03 \\
		      & 2.5e+01 & 2.7e+02 & 7.6e+02 & 1.5e+03 & 2.6e+03 & 1.5e+04 \\ \hline
	$j=100$ & 2.8e+00 & 5.8e+01 & 1.6e+02 & 3.3e+02 & 5.7e+02 & 3.2e+03 \\
		      & 4.2e+01 & 3.8e+02 & 1.1e+03 & 2.2e+03 & 3.7e+03 & 2.1e+04 \\ \hline
	$j=200$ & 3.2e+00 & 5.8e+01 & 1.6e+02 & 3.3e+02 & 5.7e+02 & 3.2e+03 \\
		      & 7.2e+01 & 5.4e+02 & 1.5e+03 & 3.0e+03 & 5.3e+03 & 3.0e+04 \\ \hline
	$j=400$ & 4.5e+00 & 5.8e+01 & 1.8e+02 & 3.3e+02 & 5.7e+02 & 3.6e+03 \\
	  	    & 1.3e+02 & 7.3e+02 & 2.3e+03 & 4.0e+03 & 6.9e+03 & 4.2e+04 \\
\hline
\end{tabular}
\caption{Propagation matrix maximum norm $\|\bar{G}^{-1}_{j}\|_{\max}$ (top) and the bound $\|(P_l(\bar{H}_{j,j}))^{-1}\|$  (bottom), see \eqref{eq:ineq1}, for p($\ell$)-CG solution of a 2D Poisson problem with 200 $\times$ 200 unknowns with different pipeline lengths $l$ and in various iterations $j$. Shifts $\sigma_i$ for p($\ell$)-CG are based on the degree $l$ Chebyshev polynomial on the interval [0,8]. The results displayed in the table correspond to the convergence history shown in the left panel of Fig.~\ref{fig:figure1}.}
\label{tab:bounds}
\end{table}

Table \ref{tab:bounds} compares the maximum norm of $\bar{G}^{-1}_{j}$ in p($\ell$)-CG to the theoretical upper bound $\|(P_l(\bar{H}_{j,j}))^{-1}\|$ (see Section \ref{sec:further}, Lemma \ref{lemma:lemma_bound}) corresponding to the convergence histories shown in Figure \ref{fig:figure1}. The bound is relatively tight, differing at most two orders of magnitude from $\|\bar{G}^{-1}_{j}\|_{\max}$. Note that the size of $\|\bar{G}^{-1}_{j}\|_{\max}$ and the bound $\|(P_l(\bar{H}_{j,j}))^{-1}\|$ increase with respect to both the pipeline length $l$ (see Remark \ref{remark:depth}) and the iteration index $j$ (see Remark \ref{remark:index}).

\begin{figure}
\begin{center}
\includegraphics[height=0.35\textwidth]{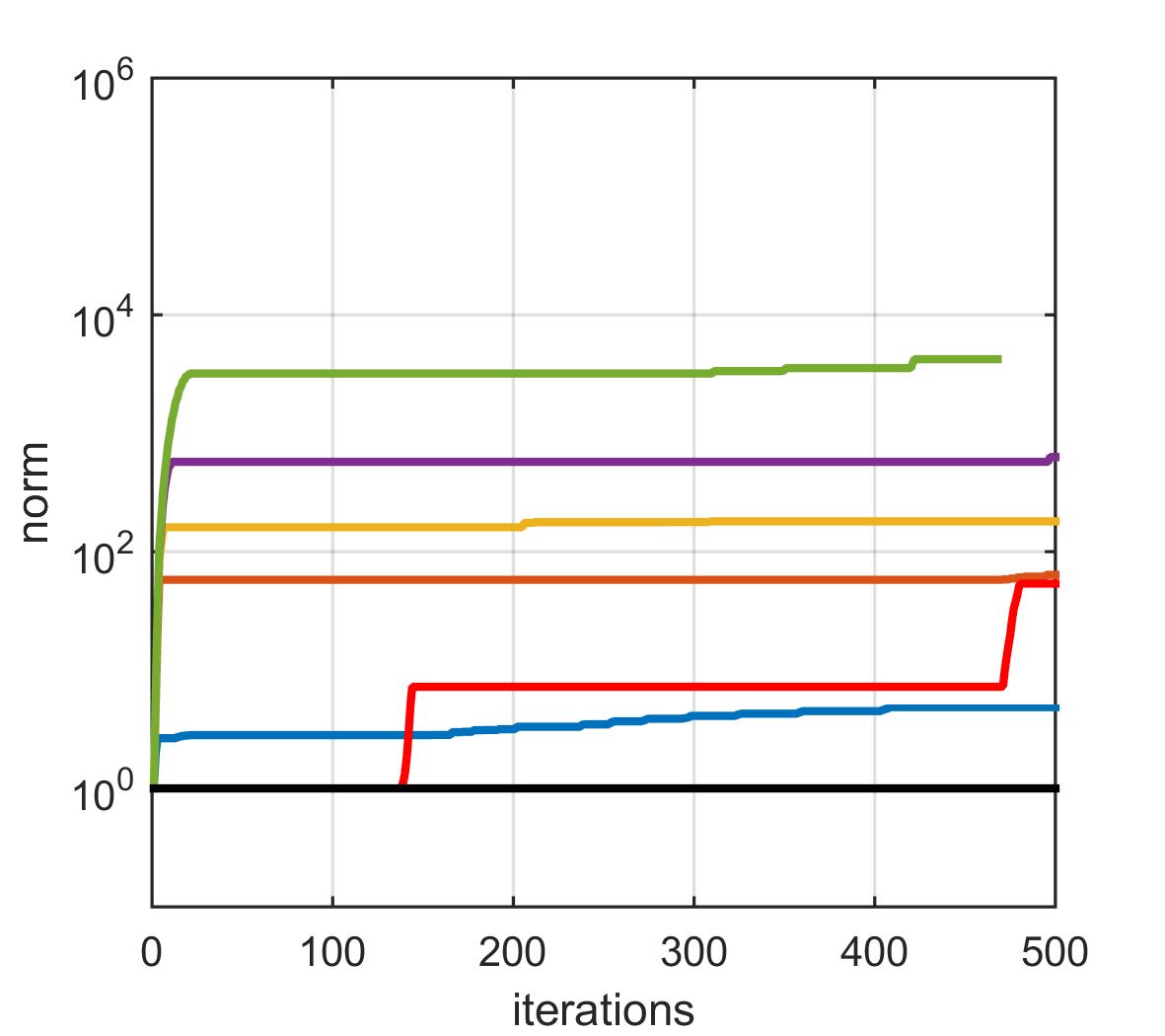} 
\includegraphics[height=0.35\textwidth]{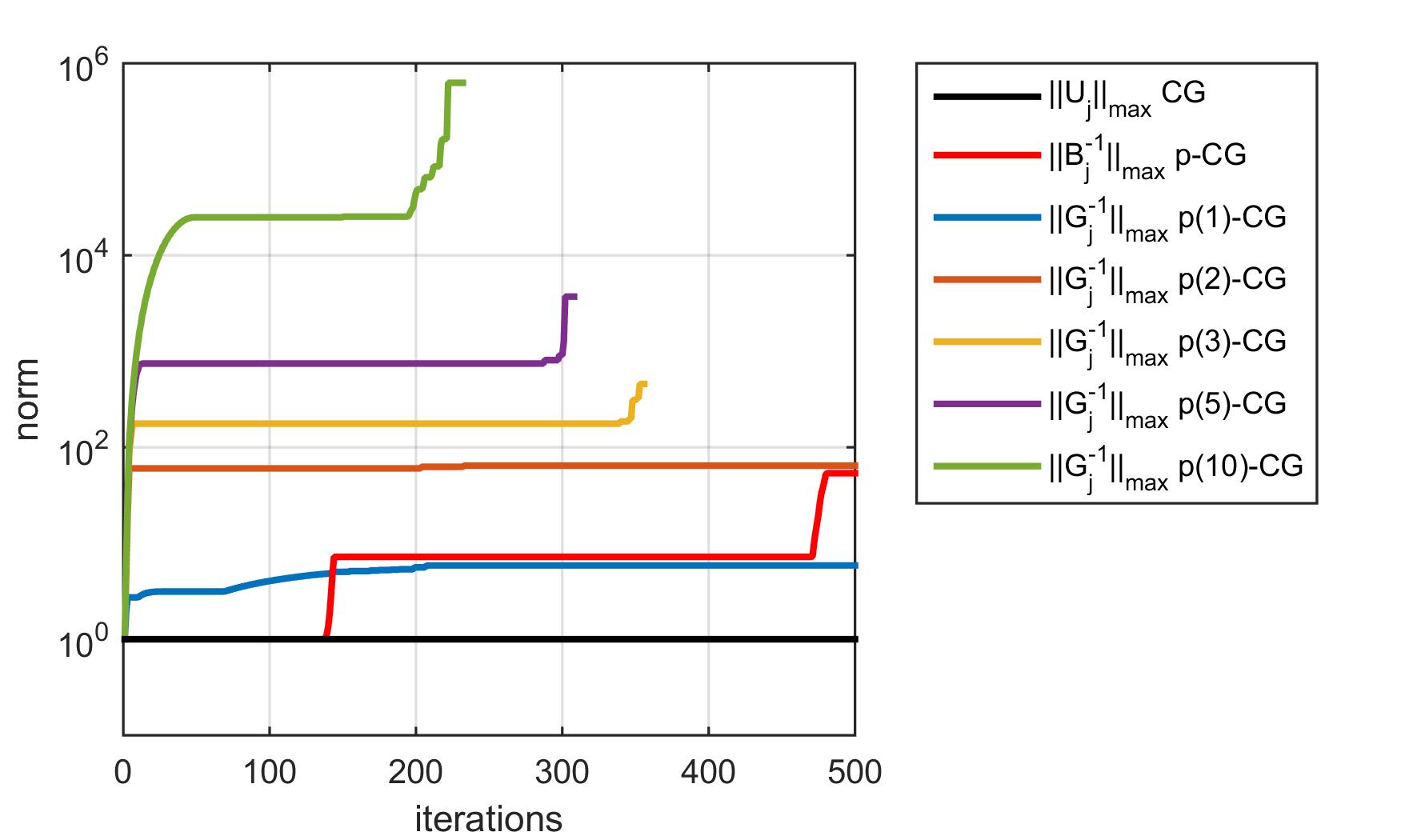} 
\end{center}
\caption{Maximum norms of the essential matrices $U_j$, $\bar{B}_j^{-1}$ and $\bar{G}_j^{-1}$ involved in the local rounding error propagation for different variants of CG. See expression \eqref{eq:f_CG_mat} for CG, \eqref{eq:total_gap_pcg} for p-CG and \eqref{eq:gap_plcg} for p($\ell$)-CG respectively. Left: with optimal Chebyshev-based shifts on the interval [0,8], cf.~Fig.~\ref{fig:figure1} (left). Right: with sub-optimal Chebyshev shifts on the interval [0,8*1.005], cf.~Fig.\ref{fig:figure1} (right).}
\label{fig:figure3}
\end{figure}

In Figure \ref{fig:figure3} the maximum norm $\|\bar{G}_j^{-1}\|_{\max}$ is shown as a function of the iteration $j$ for different pipeline lengths $l$. The maximum norms $\|U_{j}\|_{\max}$ for CG, see \eqref{eq:f_CG_mat}, and $\|\mathcal{B}^{-1}_j\|_{\max}$ for p-CG, see \eqref{eq:total_gap_pcg}, are also displayed. The impact of increasing pipeline lengths on numerical stability is clear from the figure. A comparison between the left panel (optimal shifts) and the right panel (sub-optimal shifts) in Figure \ref{fig:figure3} illustrates the influence of the basis choice on the norm of $\bar{G}_j^{-1}$, see Remark \ref{remark:shifts}. Note that no data is plotted when the matrix $\bar{G}_j$ becomes numerically singular, which corresponds to iterations in which a square root breakdown occurs due to numerical round-off, see Figure \ref{fig:figure1}.

Relating Figure \ref{fig:figure3} to the corresponding convergence histories in Figure \ref{fig:figure1} for this problem, 
it is clear that the maximal attainable accuracy for p-CG is comparable to that of p(2)-CG, whereas the p(1)-CG algorithm is able to attain a better final precision. A similar observation was made in \cite{cornelis2017communication} without the explanatory numerical analysis for the given benchmark problem. The final accuracy level at which the residual stagnates degrades significantly as longer pipelines are used.

\begin{figure}
\begin{center}
\includegraphics[width=0.47\textwidth]{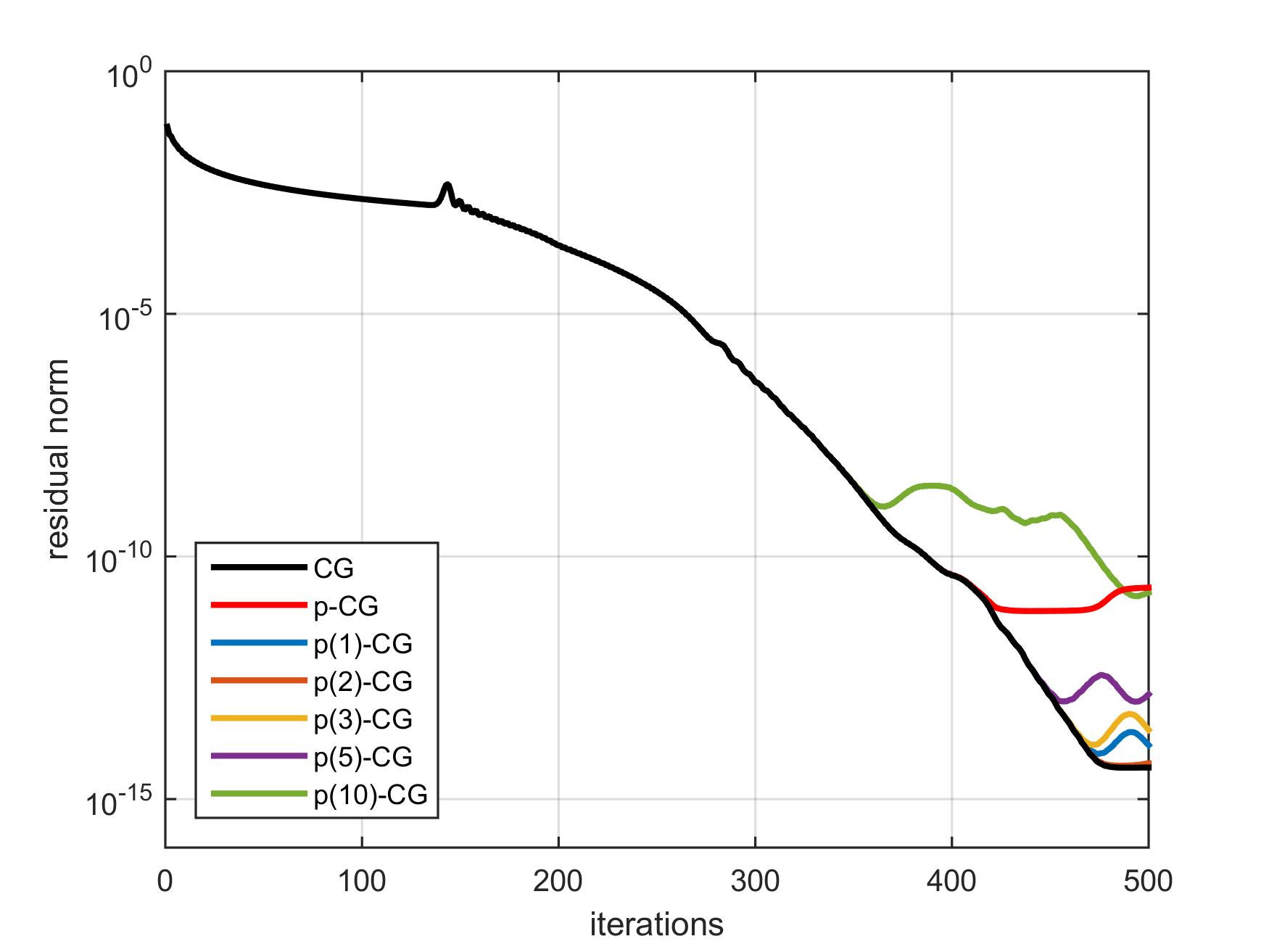}
\includegraphics[width=0.47\textwidth]{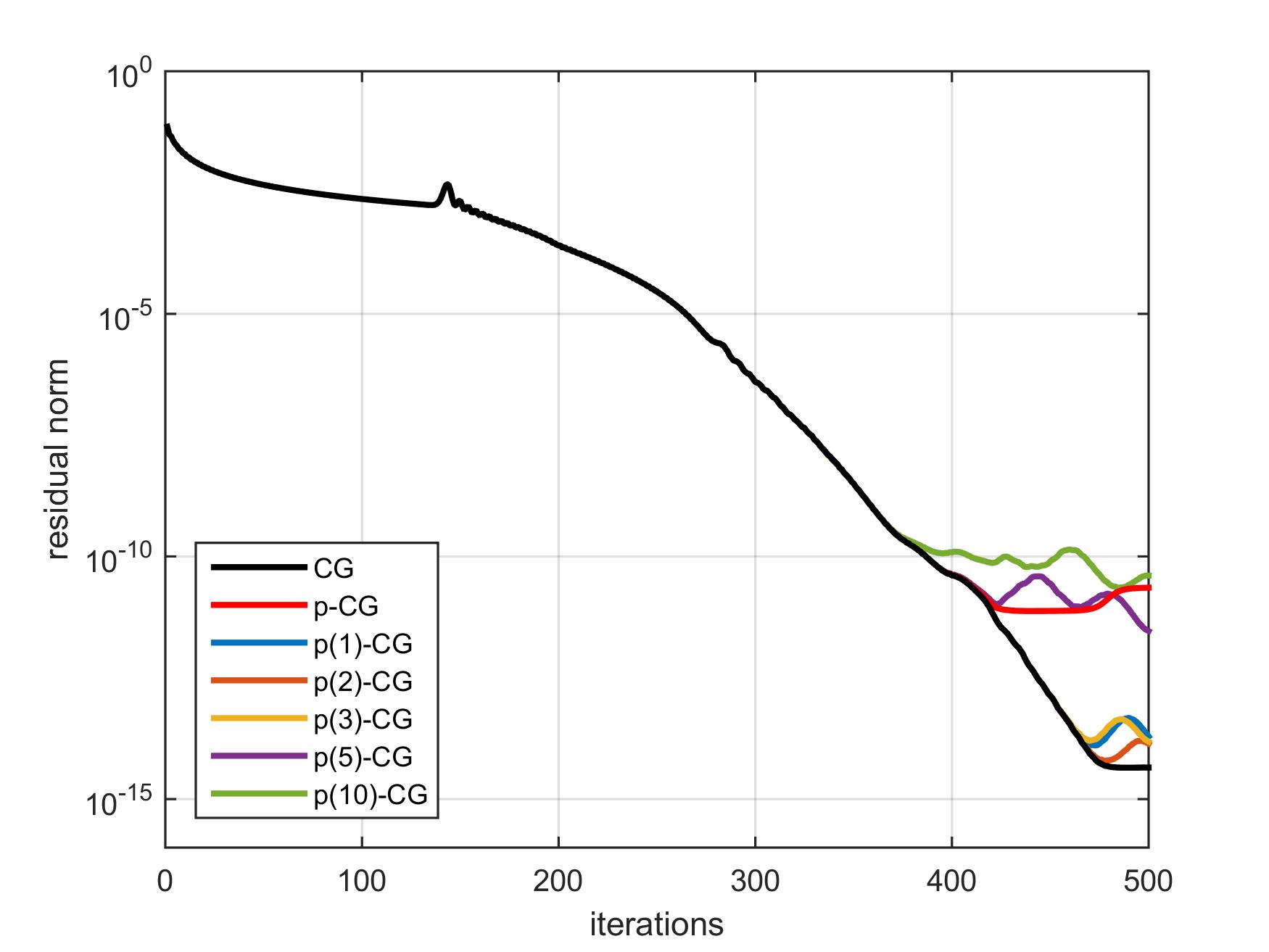}
\end{center}
\caption{Comparison of true residual norm history $\|b-A\bar{x}_j\|$ for different CG variants on a 2D Poisson problem with 200 $\times$ 200 unknowns. The p($\ell$)-CG variants use the recurrence \eqref{eq:vbar_stab} to improve numerical stability by avoiding rounding error accumulation. The stabilizing shifts $\sigma_i$ for p($\ell$)-CG are based on the degree $l$ Chebyshev polynomial on the interval [0,8] (left) and the interval [0,8*1.005] (right).}
\label{fig:figure4}
\end{figure}

\begin{figure}
\begin{center}
\begin{tabular}{cc}
\includegraphics[width=0.47\textwidth]{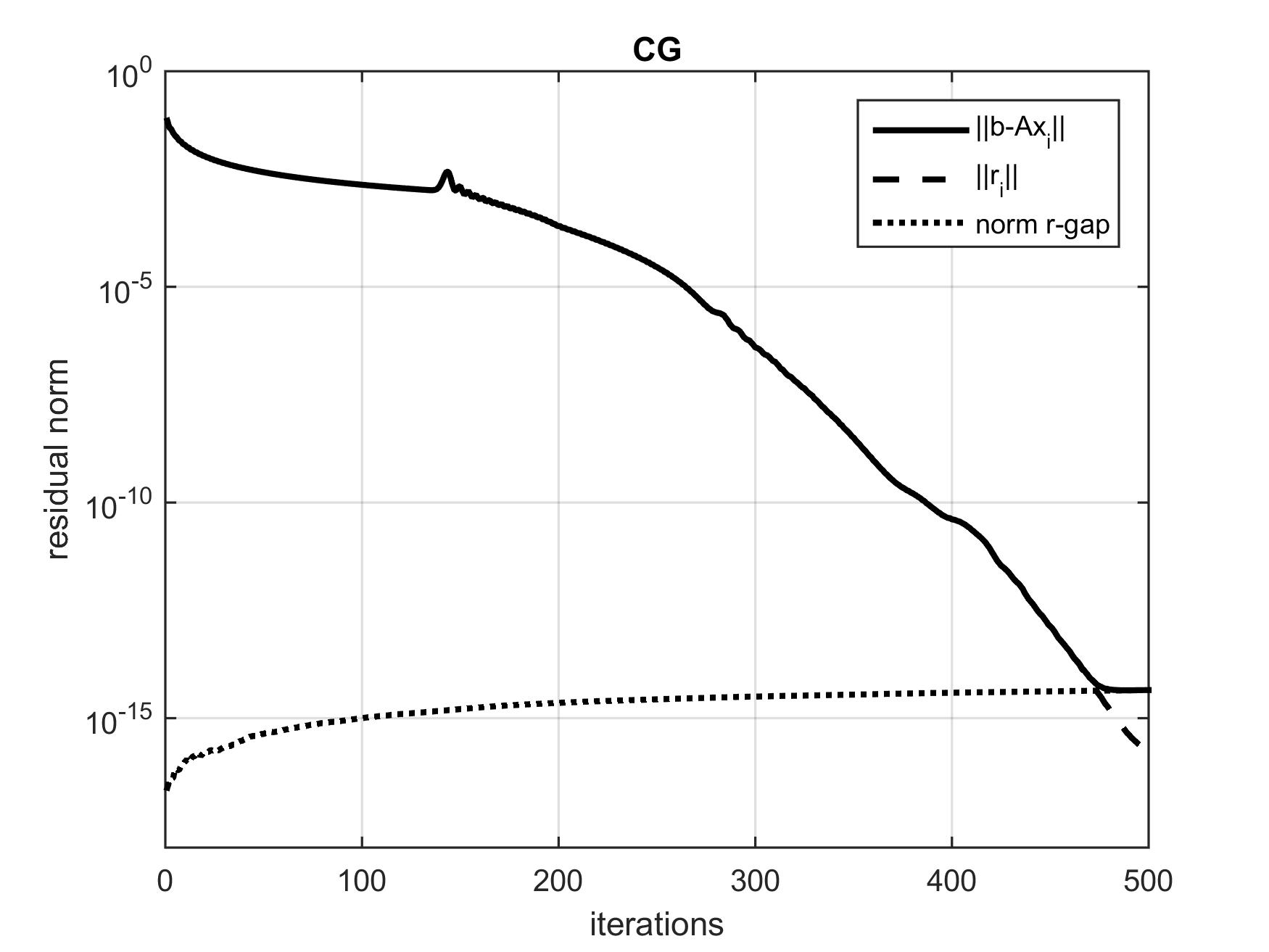} &
\includegraphics[width=0.47\textwidth]{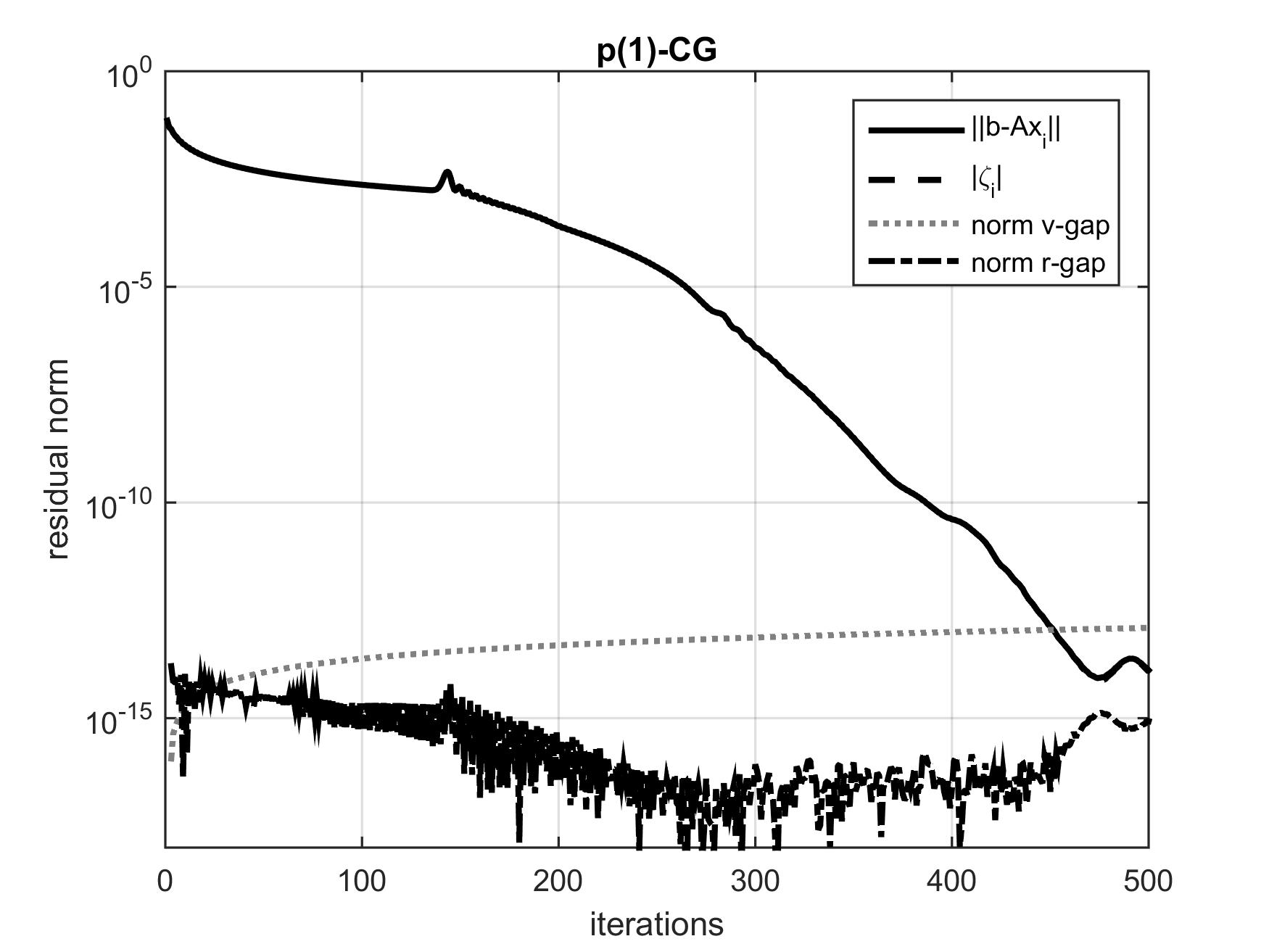} \\
\includegraphics[width=0.47\textwidth]{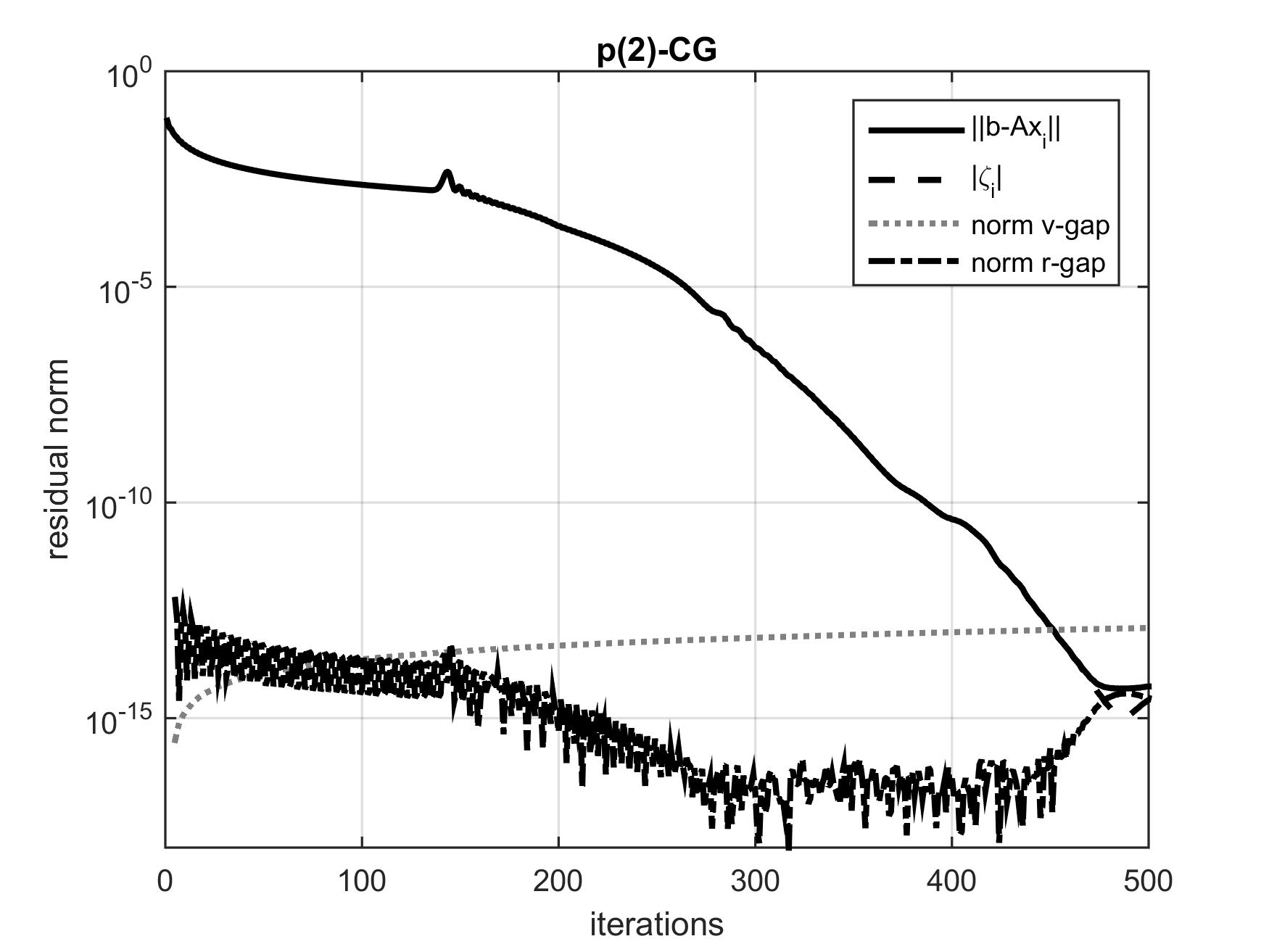} &
\includegraphics[width=0.47\textwidth]{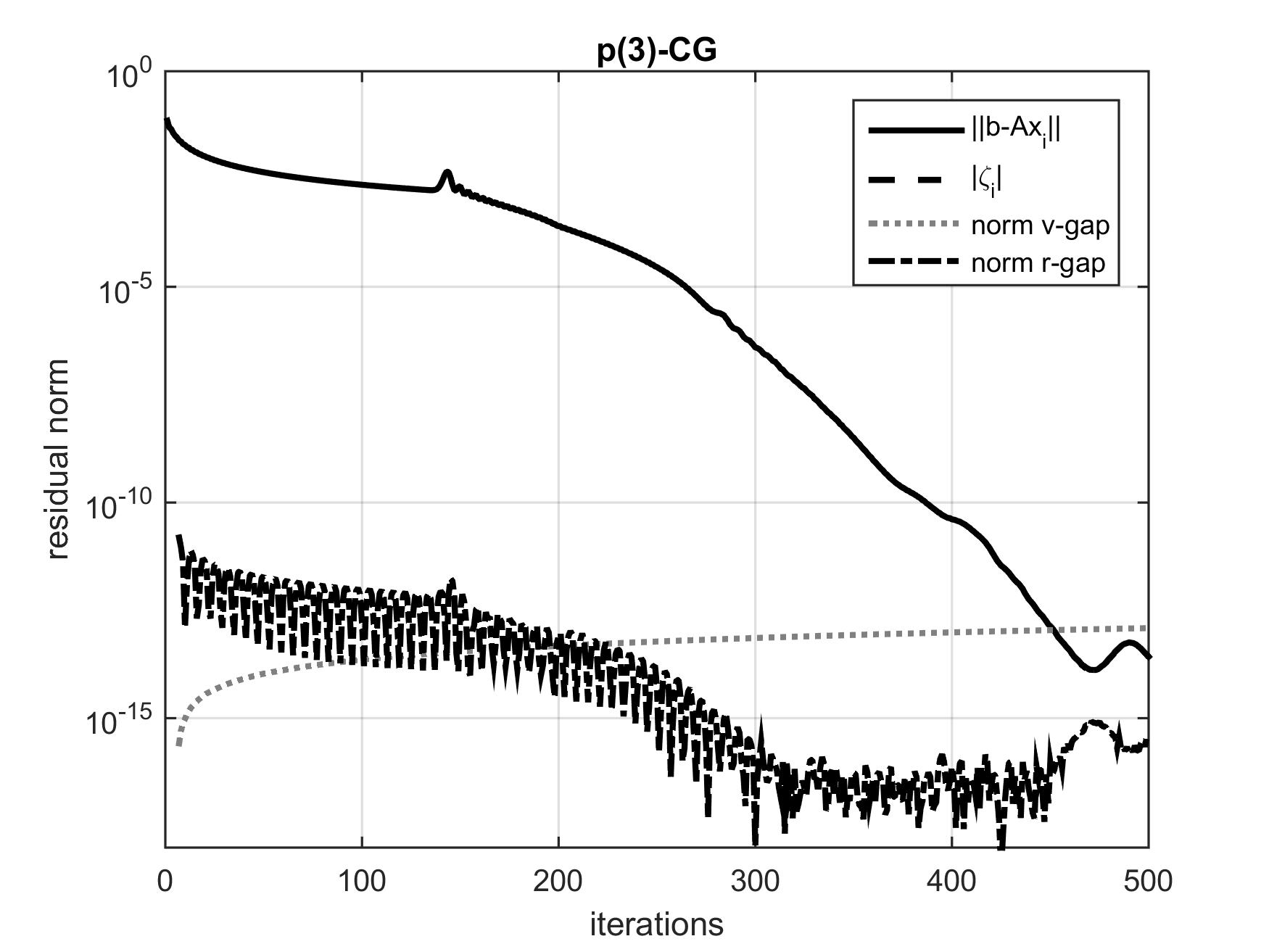} \\
\includegraphics[width=0.47\textwidth]{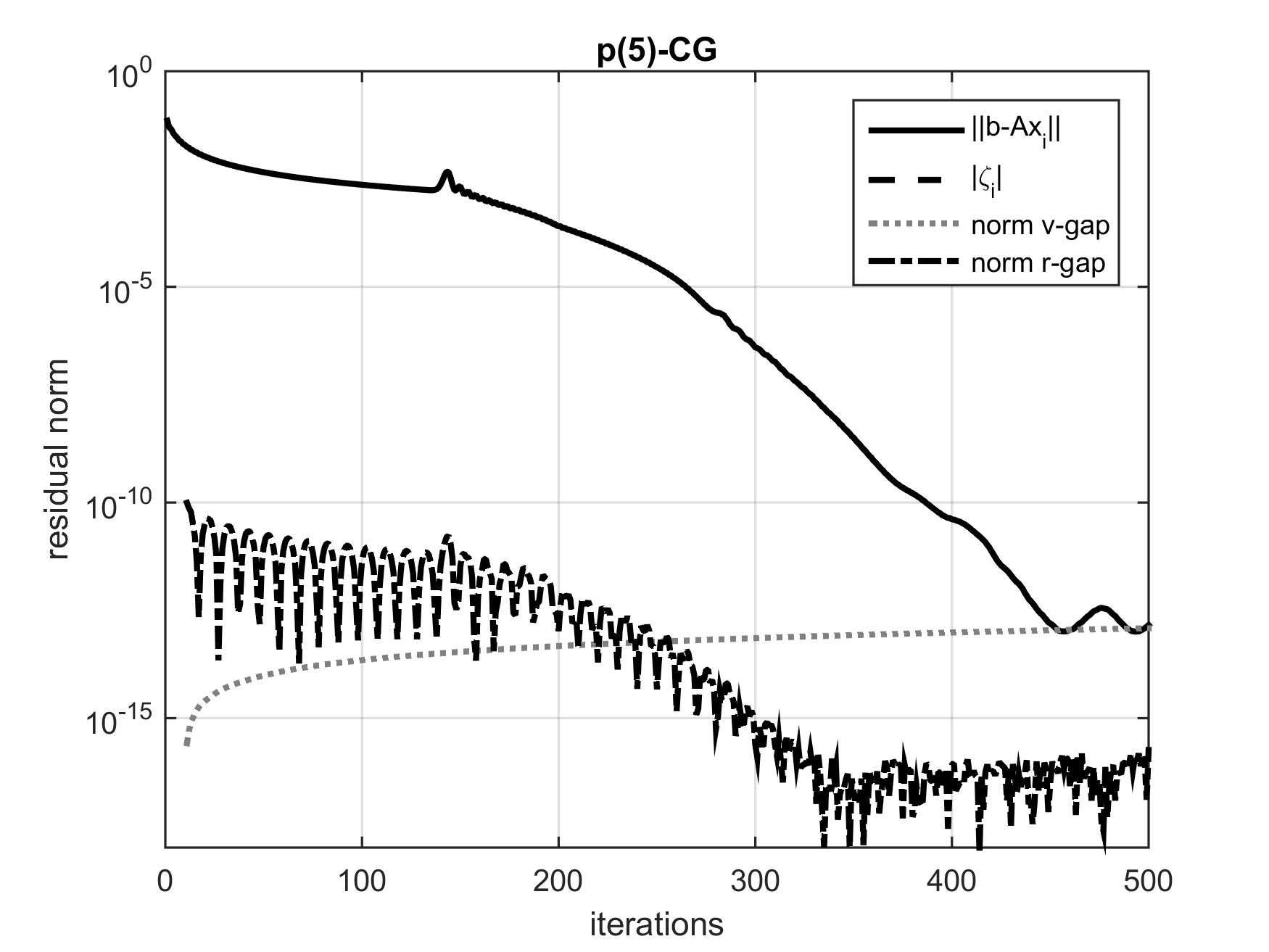} &
\includegraphics[width=0.47\textwidth]{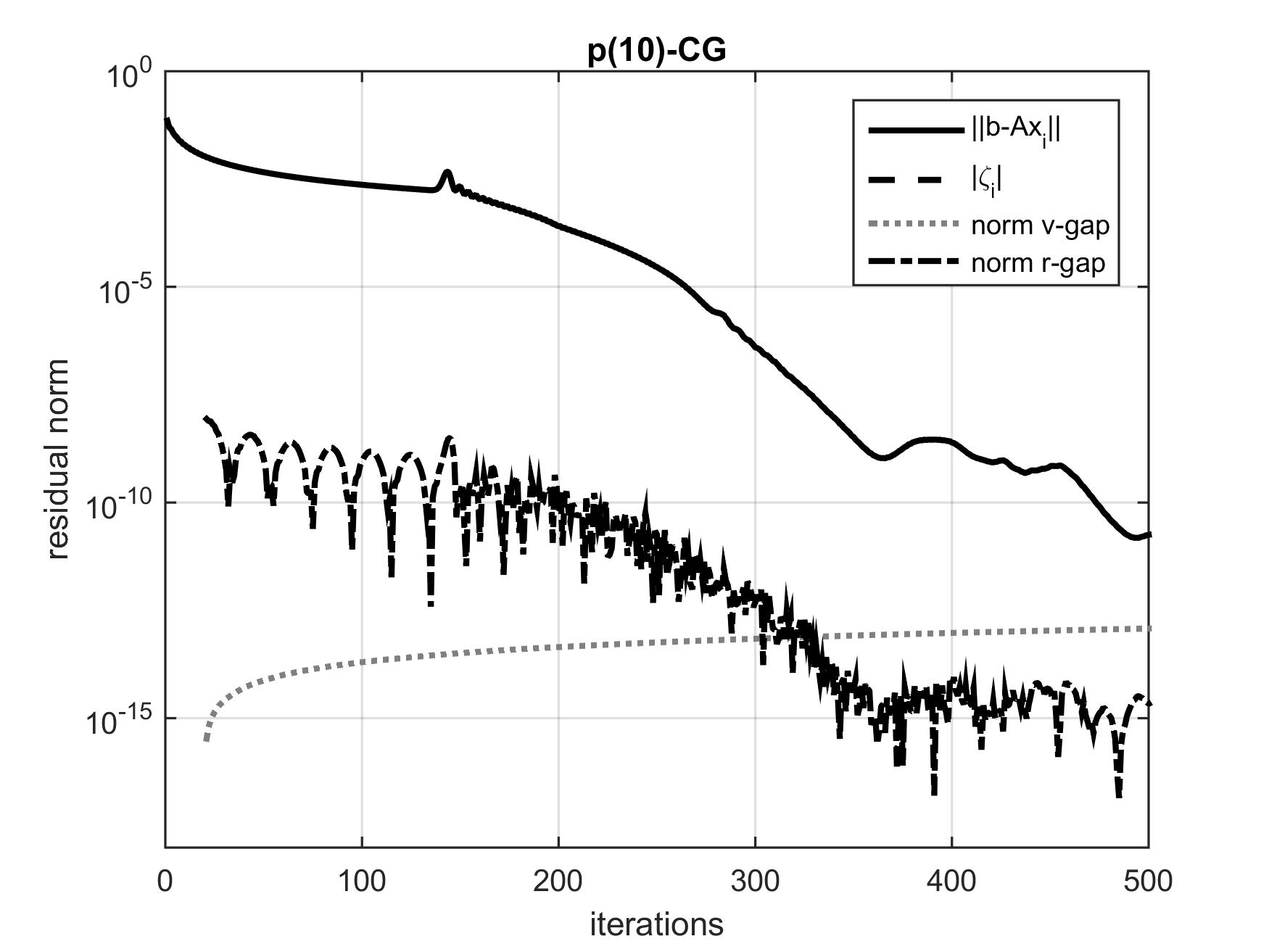} 
\end{tabular}
\end{center}
\caption{True residual norm $\|b-A\bar{x}_{j+1}\|$, computed residual norm $\|\bar{r}_{j+1}\|$, gap norm $\|f_{j+1}\| = \|\bar{\bold{v}}_{j+1} - \bar{v}_{j+1}\|$ and gap norm $\|\bar{\bold{r}}_{j+1}-\bar{r}_{j+1}\|$ for different CG variants on a 2D Poisson problem with 200 $\times$ 200 unknowns corresponding to Fig.~\ref{fig:figure4} (left). For CG (top left) the norm of the residual gap $f_{j+1} = (b-A\bar{x}_{j+1}) - \bar{r}_{j+1}$ is displayed, where $\bar{r}_{j+1}$ is computed using \eqref{eq:recs_xandr}. For p($\ell$)-CG with $l = 1,2,3,5,10$ the norms of the gaps $\bar{\bold{v}}_{j-l+1} - \bar{v}_{j-l+1}$ and $\bar{\bold{r}}_{j-l+1} - \bar{r}_{j-l+1}$ are shown, where $\bar{\bold{v}}_{j-l+1}$ satisfies \eqref{eq:v_arnoldi}, $\bar{v}_{j-l+1}$ is computed using the recurrence relation \eqref{eq:vbar_stab} for increased numerical stability,  and $\bar{\bold{r}}_{j-l+1} - \bar{r}_{j-l+1}$ is defined by \eqref{eq:resgap}. See Fig.~\ref{fig:figure2} for comparison with original p($\ell$)-CG. For p(10)-CG delayed convergence due to loss of Krylov basis orthogonality is observed.}
\label{fig:figure5}
\end{figure}

Figures \ref{fig:figure4} and \ref{fig:figure5} are the analogue of Figures \ref{fig:figure1} and \ref{fig:figure2}, where the recurrence relation \eqref{eq:vbar_stab} for $\bar{v}_{j+1}$ is used to improve the stability of the algorithm with respect to local rounding error propagation. With this variant of the algorithm the gaps between $\bar{\bold{v}}_j$ and $\bar{v}_j$ reduce to local rounding errors for any pipeline length $l$, cf.~\eqref{eq:gap_stab}, as illustrated by Figure \ref{fig:figure5}, and maximal attainable accuracy of p($\ell$)-CG is improved significantly for all values of $l$. Note that contrary to Figure \ref{fig:figure1} no square root breakdowns occur in stabilized p($\ell$)-CG for any of the choices of $l$ shown in Figure \ref{fig:figure4}.

\begin{figure}
\begin{center}
\includegraphics[width=0.47\textwidth]{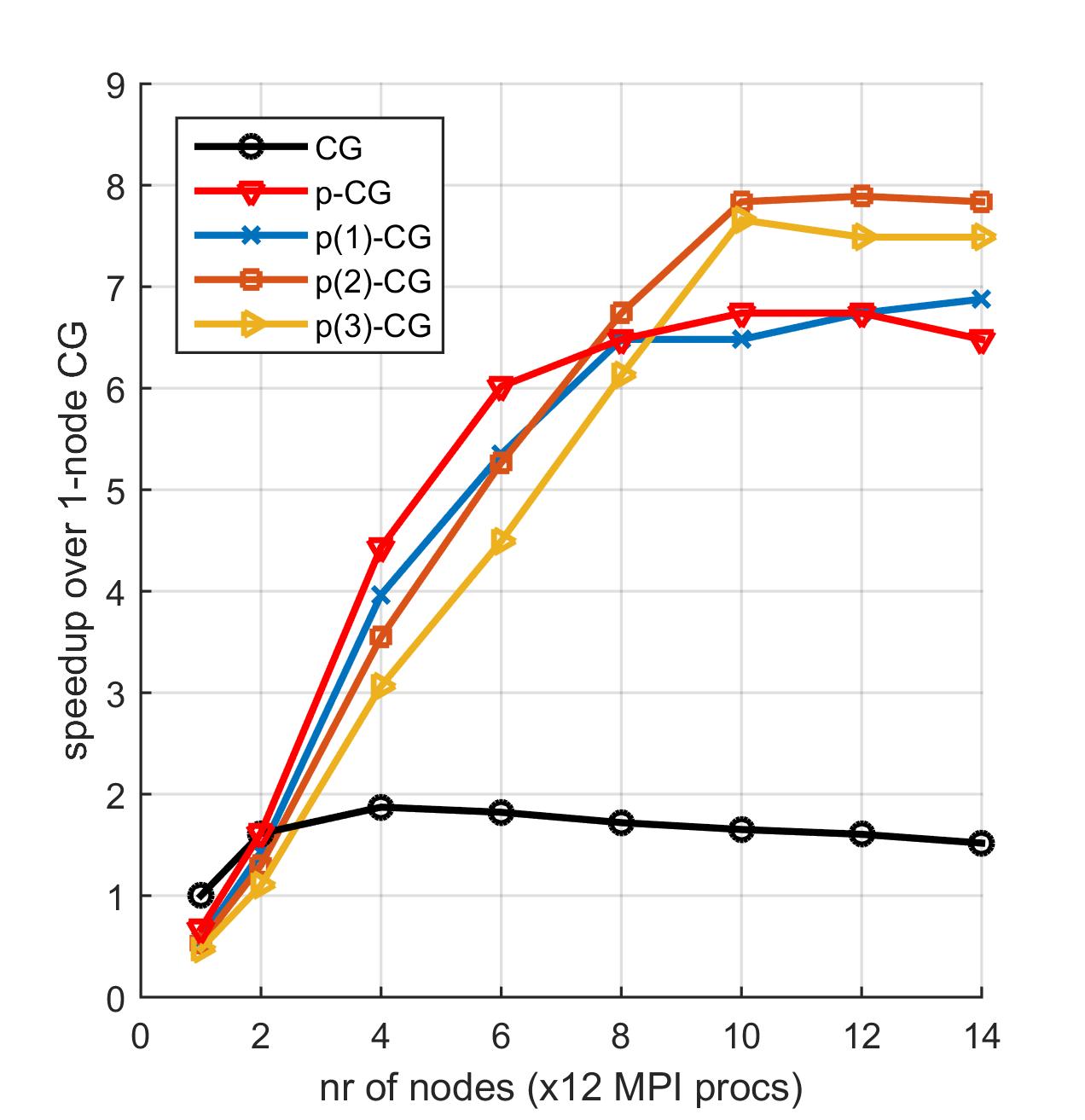}
\includegraphics[width=0.47\textwidth]{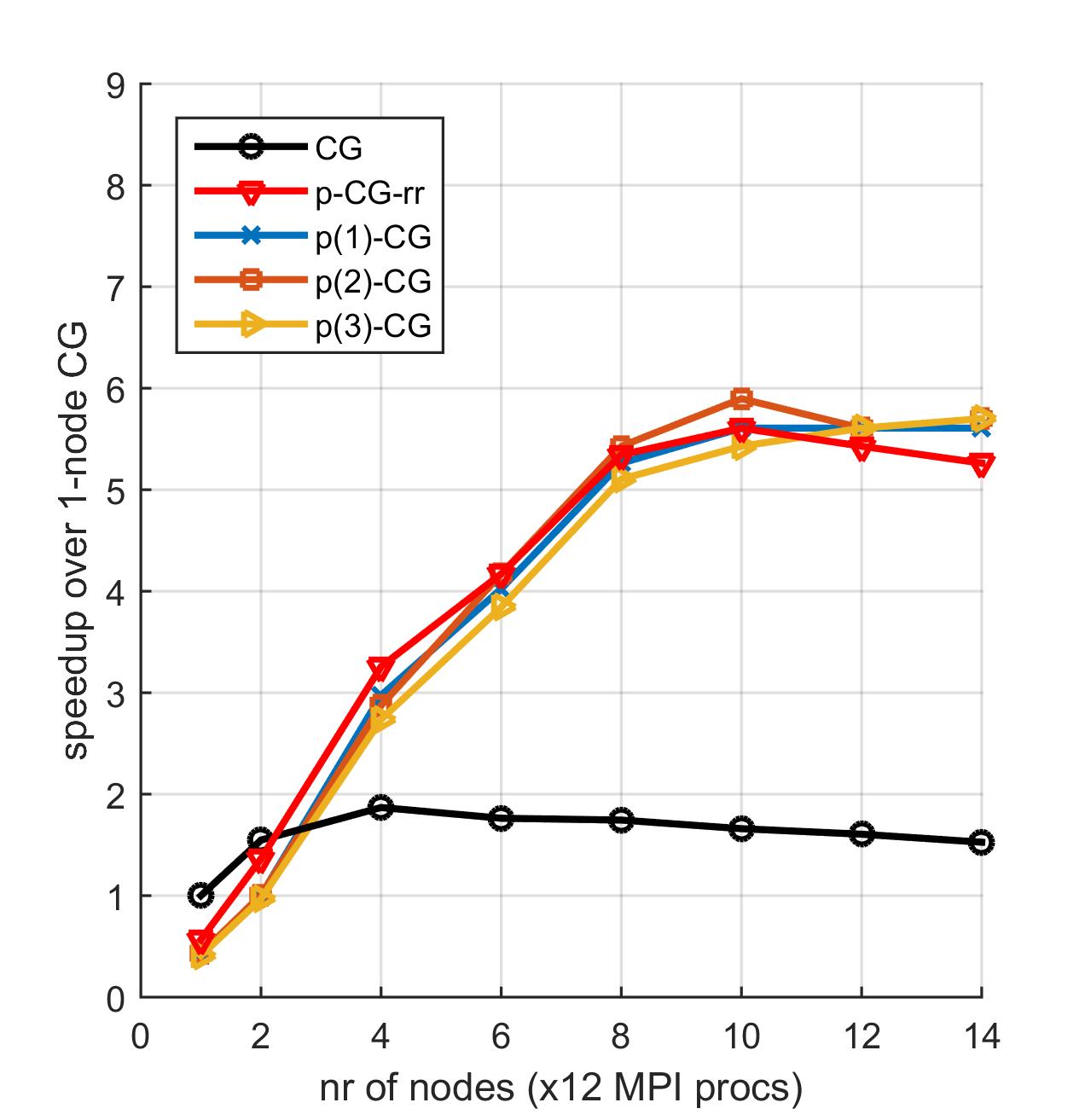}
\end{center}
\caption{Strong scaling experiment on up to 14 nodes (168 processes) for a 5-point stencil 2D Poisson problem with $750 \times 750$ unknowns. Speedup over single node classic CG for various pipelined CG variants. All method converged to $\|\bar{r}_j\|/\|b\| = 1.0\text{e-}5$ in 1019 iterations. Left: speedup of p-CG \cite{ghysels2014hiding} and p($\ell$)-CG with the standard recurrence relation \eqref{eq:vbar_rec} for $\bar{v}_{j+1}$. Right: speedup of the numerically more stable p-CG-rr algorithm (pipelined CG with automated residual replacement) \cite{cools2018analyzing} and p($\ell$)-CG with the stabilized recurrence relation \eqref{eq:vbar_stab} for $\bar{v}_{j+1}$.}
\label{fig:figure6}
\end{figure}

Figure \ref{fig:figure6} shows a strong scaling experiment on a relative small cluster with $14$ compute nodes, consisting of two $6$-core Intel Xeon X5660 Nehalem $2.80$ GHz processors each (12 cores per node). Nodes are connected by $4\,\times\,$QDR InfiniBand technology (32 Gb/s point-to-point bandwidth).
The algorithms are implemented in PETSc version 3.8.3 \cite{petsc-web-page}. Communication is performed using Intel MPI 2018.1.163 based on the MPI 3.1 standard. The PETSc environment variables 
\texttt{MPICH\_ASYNC\_PROGRESS=1} and \texttt{MPICH\_MAX\_THREAD\_SAFETY=multiple} are set to ensure optimal parallelism by allowing for non-blocking global communication.
A 2D Poisson type linear system with right-hand side $b = A\hat{x}$ with $\hat{x} = 1$ is solved. The initial guess is $\bar{x}_0 = 0$ for every method.
The benchmark problem is available as example $2$ in the PETSc KSP directory. 
The simulation domain is discretized using $750 \times 750$ grid points (562,500 unknowns). 
The tolerance imposed on the scaled recursive residual norm $\|\bar{r}_j\| / \|b\|$ is $10^{-5}$. The figure reports the minimal timings over three independent runs for each variant of the CG method and MPI configuration.
 
Figure \ref{fig:figure6} compares the performance of the p($\ell$)-CG algorithm using the standard recurrence relation \eqref{eq:vbar_rec} for $\bar{v}_{j+1}$ (left panel) to the same algorithm that uses the stabilized recurrence relation \eqref{eq:vbar_stab} for $\bar{v}_{j+1}$ (right panel). Using the original recurrence relation \eqref{eq:vbar_rec} p($2$)-CG starts out-scaling p($1$)-CG and p-CG from 8 nodes onward in this experiment. Note that p($3$)-CG does not further improve scalability, indicating that the overlap is optimal for $l = 2$. The computational cost of the additional \textsc{spmv} required to compute the stabilized recurrence relation \eqref{eq:vbar_stab} clearly affects the timings. The balance between time spent in global communication and local computation is shifted more towards the computational side, which implies that in the stabilized p($\ell$)-CG algorithm the use of deeper pipelines ($l > 1$) is no longer beneficial for this test case. However, using deeper pipelines in combination with the stabilized recurrence relation \eqref{eq:vbar_stab} would likely be useful if the computational work per processor was further reduced.

\begin{figure}
\begin{center}
\includegraphics[width=0.47\textwidth]{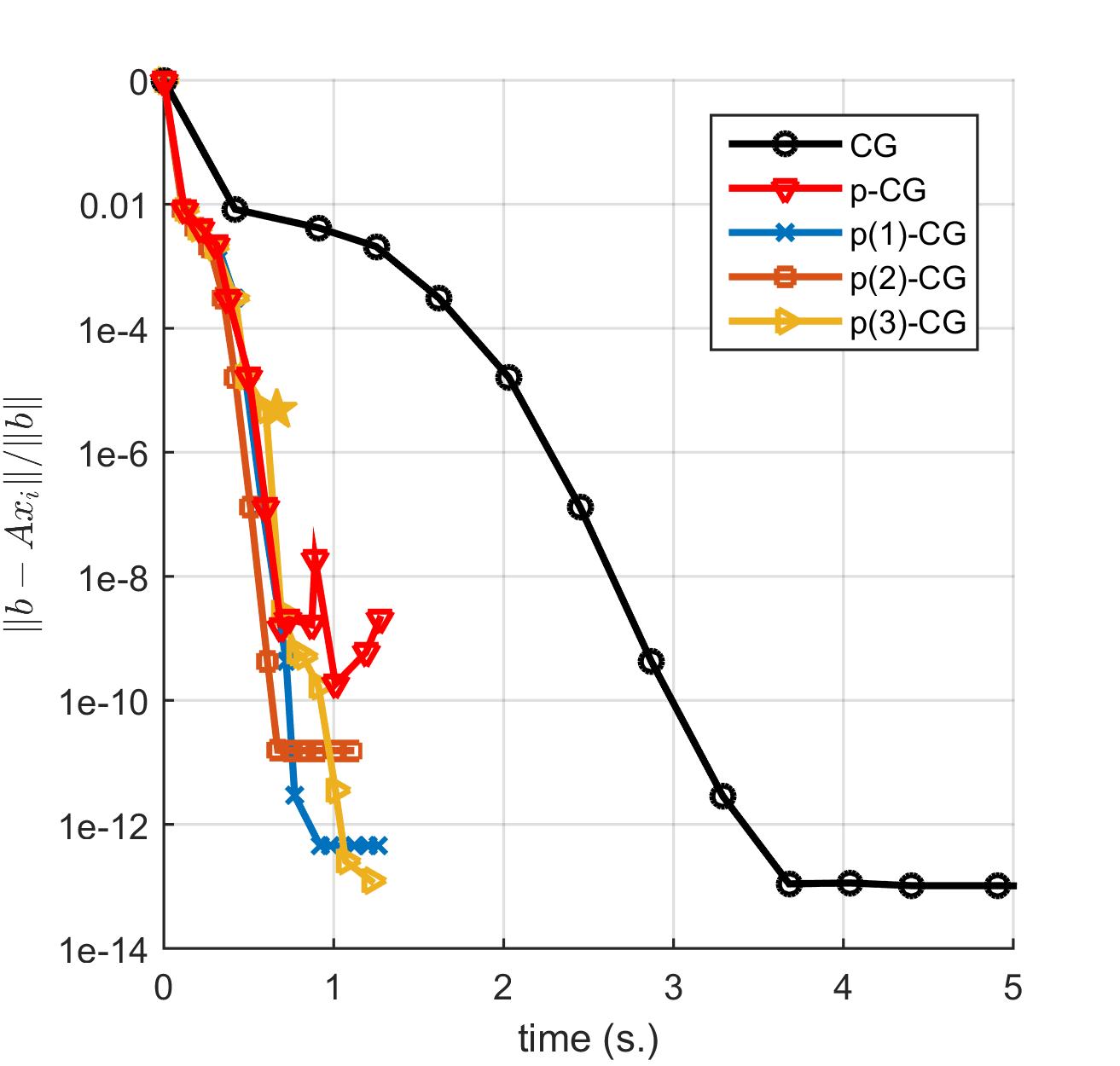}
\includegraphics[width=0.47\textwidth]{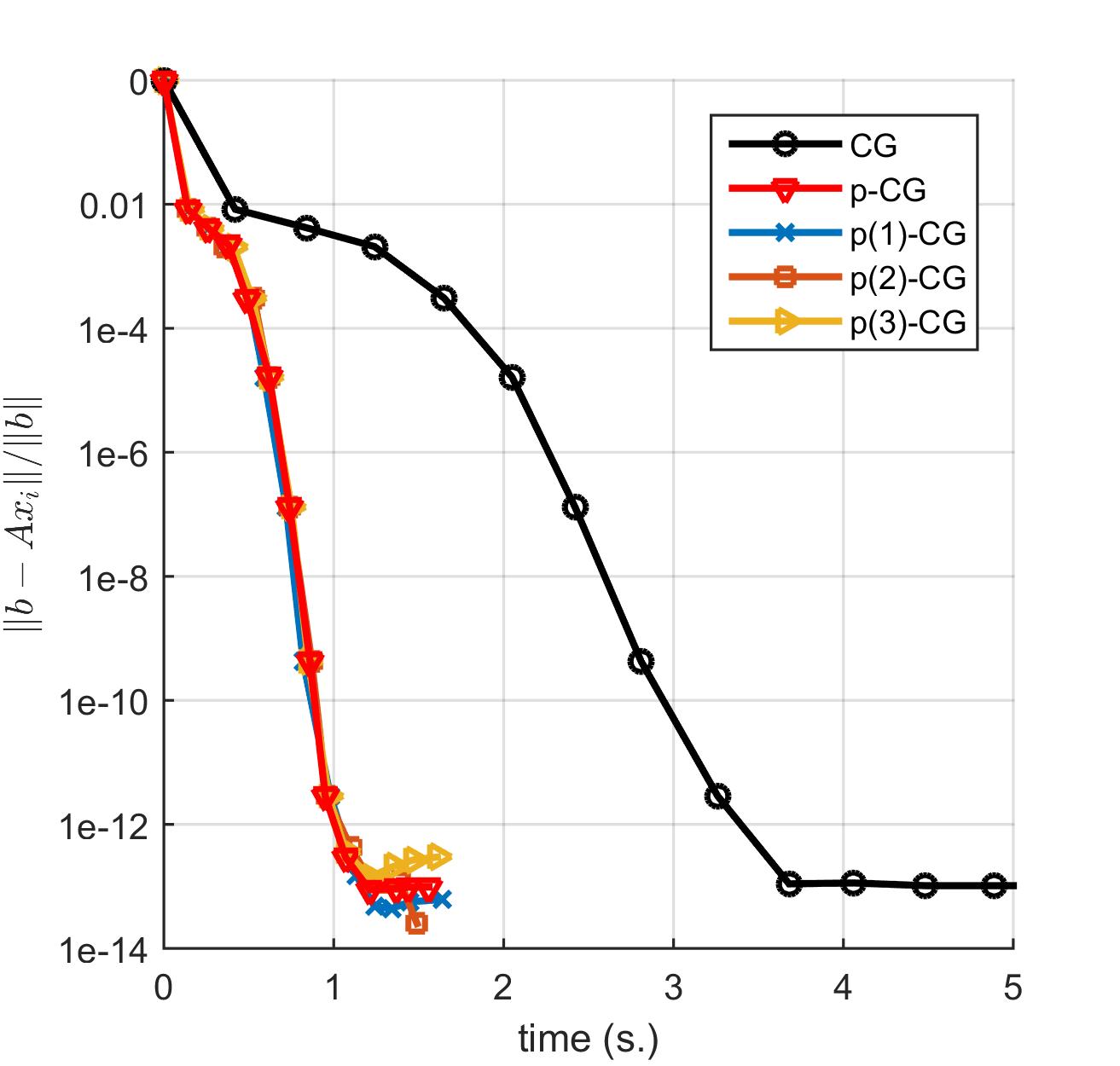}
\end{center}
\caption{Accuracy experiment on 10 nodes (120 processes) for a 5-point stencil 2D Poisson problem with $750 \times 750$ unknowns. Residual norm $\|b - A \bar{x}_j\|$ as a function of total time spent by the algorithm for various pipelined CG variants. Left: attainable accuracy for p-CG \cite{ghysels2014hiding} and p($\ell$)-CG with the standard recurrence relation \eqref{eq:vbar_rec} for $\bar{v}_{j+1}$.  Square root breakdown in p(3)-CG is indicated by a $\bigstar$ symbol (followed by explicit iteration restart). Right: attainable accuracy for the stabilized p-CG-rr algorithm (pipelined CG with automated residual replacement) \cite{cools2018analyzing} and p($\ell$)-CG with the stabilized recurrence relation \eqref{eq:vbar_stab} for $\bar{v}_{j+1}$.}
\label{fig:figure7}
\end{figure}

In Figure \ref{fig:figure7} the maximal attainable accuracy of several variants to the CG method for the $750 \times 750$ 2D Laplace problem is displayed as a function of total time spent by the algorithm. This experiment is executed on 10 of the nodes specified above. The left panel shows accuracy results for p($\ell$)-CG with the standard recurrence relation \eqref{eq:vbar_rec} for $\bar{v}_{j+1}$. The p(2)-CG algorithm outperforms the other CG variants shown in the figure in terms of time to solution, cf.~Fig.~\ref{fig:figure6} (left). In the right panel the stabilized recurrence relation \eqref{eq:vbar_stab} is used. The latter ensures that p($\ell$)-CG reaches a backward error $\|b-A\bar{x}_j\|/\|b\| \leq 1.0\text{e-}12$ for $\ell = 1,2,3$, which is not attainable by p-CG, p($1$)-CG and p($2$)-CG if the recurrence relation \eqref{eq:vbar_rec} is used.\footnote{Figure \ref{fig:figure7} (left panel): Note that p($3$)-CG is able to attain a residual that satisfies $\|b-A\bar{x}_j\|/\|b\| \leq 1.0\text{e-}12$, where p($1$)-CG and p($2$)-CG do not. This is due to the square root breakdown and subsequent restart of the p($3$)-CG algorithm as described in Remark \ref{remark:sqrt_breakdown}, see also \cite{cornelis2017communication}. The restart improves final attainable accuracy but delays the convergence of the algorithm considerably compared to other pipelined variants.} However, the standard p($\ell$)-CG algorithms generally converge faster than the stabilized variants, see also Fig.~\ref{fig:figure6}.

\section{Conclusions} \label{sec:conclusions}

As HPC hardware keeps evolving towards exascale the gap between computational performance and communication latency keeps increasing. Consequently, many numerical methods that are historically optimized towards flop performance, such as Krylov subspace methods for solving large linear systems, now need to be revised towards also (or even: primarily) minimizing communication overhead. Several research teams are currently working towards this goal \cite{carson2013avoiding,mcinnes2014hierarchical,ghysels2013hiding,grigori2016enlarged,imberti2017varying,cornelis2017communication}, resulting in a variety of communication reducing variants to classic Krylov subspace methods that feature improved scalability on massively parallel hardware. However, the obtained reduction in communication does not come for free; it requires a reordering of algorithmic operations which often affects the numerical behavior of the algorithm. Hence, these communication reducing algorithms typically suffer from the propagation of round-off errors, resulting in poor attainable accuracy and delayed convergence. The numerical analysis of these recently introduced and promising HPC variants of Krylov subspace algorithms is therefore paramount.

This work focuses on analyzing the effect of local rounding errors on maximal attainable accuracy for a specific class of communication hiding Krylov subspace methods, namely pipelined Conjugate Gradient methods \cite{ghysels2014hiding,cornelis2017communication}. Pipelined CG methods aim to minimize communication overhead both by reducing the number of global synchronization points and by hiding communication latency behind useful (local) computations. However, reorganization of the algorithm introduces multi-term recurrence relations to update the solution, which are prone to local rounding error propagation \cite{carson2014residual,carson2016numerical,cools2018analyzing}. This paper characterizes the behavior of local rounding errors stemming from the multi-term recurrence relations in pipelined CG (p-CG) from \cite{ghysels2014hiding} and $\ell$-depth pipelined CG (p($\ell$)-CG) from \cite{cornelis2017communication} and analyzes the impact of rounding error amplification in pipelined methods on their attainable accuracy. Furthermore, practical bounds for the propagation of the local rounding errors are derived, which lead to insights concerning the influence of the pipelined length and the Krylov basis on maximal attainable accuracy. Following the analysis a possible countermeasure to the local rounding error amplification is suggested. This strategy trades computational efficiency for improved accuracy and might prove useful for applications requiring a high precision solution. The analysis in this work is illustrated by a number of easy-to-reproduce numerical tests and parallel performance experiments using a C implementation of the pipelined Krylov subspace algorithms in PETSc. 

It should be noted that the analysis in this manuscript is not a complete numerical analysis of ($\ell$-depth) pipelined CG, since e.g.~the impact of loss of orthogonality due to rounding error propagation is not considered in this study. In addition, we do not be expect this analysis to be directly applicable to other pipelined Krylov subspace methods such as p($\ell$)-GMRES, although the general approach would likely show resemblance. The rounding error analysis for these novel, non-trivial classes of communication reducing methods should be meticulously performed for each individual algorithm in order to map the impact of round-off errors 
on each method. We hope that the contributions in this paper may act as a starting point and a possible groundwork for future research in this direction.

\section*{Acknowledgments}
The author gratefully acknowledges funding by the Flemish Research Foundation (FWO Flanders) under grant 12H4617N. 
Additionally, the author is grateful to Wim Vanroose and Jeffrey Cornelis (UAntwerp, BE) and Pieter Ghysels (LBNL, US) for interesting discussions on the topic. 
The author would also like to cordially thank the anonymous referees of the current and former versions of this manuscript for their useful comments and valuable suggestions.

\bibliographystyle{plain}
\bibliography{refs2}

\end{document}